\documentclass[oneside,english]{amsart}
\usepackage[latin9]{inputenc}
\usepackage{geometry}
\usepackage{amstext}
\usepackage{amsthm}
\usepackage{amssymb}
\usepackage{hyperref}

\usepackage[dvipsnames]{xcolor}

\definecolor{b}{HTML}{4472c4}
\definecolor{o}{HTML}{ED7D31}
\definecolor{g}{HTML}{70ad47}

\definecolor{t}{RGB}{40,154,150}

\newcommand\co[1]{{#1}}

\newtheorem{theor}{Theorem}[section]
\newtheorem*{theorem*}{Theorem}

\newtheorem{lemma}[theor]{Lemma}
\newtheorem{cor}[theor]{Corollary}
\newtheorem{defi}[theor]{Definition}
\newtheorem{prop}[theor]{Proposition}

\theoremstyle{definition}
\newtheorem{rem}[theor]{Remark}

\theoremstyle{plain}
\newtheorem*{conj*}{Conjecture}

\newcommand{\R}{\mathbb{R}}

\def\Prob{{\mathbb P}}

\def\dist{{\rm dist}}

\title[Dimension-dependent concentration for convex Lipschitz functions]
{On dimension-dependent concentration for convex Lipschitz functions in product spaces}

\author{Han Huang
\email{hhuang421@gatech.edu}}

\author{Konstantin Tikhomirov
\email{ktikhomirov6@gatech.edu}}
\thanks{K.T. is partially supported by the Sloan Research Fellowship}

\def\R{{\mathbb R}}

\def\Med{{\rm Med}}

\def\Prob{{\mathbb P}}
\def\Exp{{\mathbb E}}

\newcommand{\Var}{{\rm Var}}

\def\dist{{\rm dist}}

\date{\today}

\begin{document}

\maketitle

\begin{abstract}
Let $n\geq 1$, $K>0$, and let $X=(X_1,X_2,\dots,X_n)$ be a random vector in $\R^n$ with independent $K$--subgaussian
components.
We show that for every $1$--Lipschitz convex function $f$ in $\R^n$ (the Lipschitzness with respect to the Euclidean metric),
$$
\max\big(\Prob\big\{f(X)-\Med\,f(X)\geq t\big\},\Prob\big\{f(X)-\Med\,f(X)\leq -t\big\}\big)\leq \exp\bigg(
-\frac{c\,t^2}{K^2\log\big(2+\frac{K^2 n}{t^2}\big)}\bigg),\quad t>0,
$$
where $c>0$ is a universal constant. The estimates are optimal in the sense that for every $n\geq \tilde C$ and $t>0$
there exist a product probability distribution $X$ in $\R^n$ with $K$--subgaussian components, and
a $1$--Lipschitz convex function $f$, with
$$
\Prob\big\{\big|f(X)-\Med\,f(X)\big|\geq t\big\}\geq \tilde c\,\exp\bigg(
-\frac{\tilde C\,t^2}{K^2\log\big(2+\frac{K^2n}{t^2}\big)}\bigg).
$$
The obtained deviation estimates for subgaussian variables are in sharp contrast with the case of variables with bounded
$\|X_i\|_{\psi_p}$--quasi-norms for $p\in(0,2)$.
\end{abstract}

\section{Introduction}

Concentration in product probability spaces is an active research direction with numerous available results
(see, in particular, monographs \cite{Ledoux,BLM}).
Among classical examples of such results are Bernstein-type inequalities \cite[Chapter~2]{BLM}
for linear combinations of independent random variables, and the isoperimetric inequality in the Gauss space
which implies subgaussian dimension-free concentration \cite{ST,BorGauss} (see also \cite{Ehrhard,BakryLedoux,BobkovGauss,BartheMaurey}
as well as \cite[Theorem~V.1]{MS}).

Let $(\Omega_i,\Sigma_i,\mu_i)$, $i\geq 1$, be probability spaces, and for a given $n\geq 1$, let $\mathcal F_n$
be a subset of real valued measurable functions $f$ on the product space $(\prod_{i=1}^n \Omega_i,
\prod_{i=1}^n\Sigma_i,\mu_1\times\dots\times \mu_n)$. A question is to estimate for every $t>0$ the quantity
\begin{equation}\label{31941313049}
\sup\limits_{f\in\mathcal F_n}\;\max\big((\mu_1\times\dots\times \mu_n)\big\{f-\Med\, f\geq t\big\},
(\mu_1\times\dots\times \mu_n)\big\{f-\Med\, f\leq -t\big\}\big)
\end{equation}
(we focus on deviation from the median; see, for example,
\cite[Propositions~1.7,~1.8]{Ledoux} for relations between deviations from the mean
and the median).

\bigskip

First, let $\mathcal F_n$ be the class of $1$-Lipschitz functions \co{on} $\R^n$
(here and further in this note, the Lipschitzness is with respect to the standard Euclidean metric in $\R^n$),
and $\mu_1,\dots,\mu_n$ be Borel probability measures \co{on} $\R$.
In particular, it is known that whenever measures $\mu_i$ satisfy the {\it Poincar\'e inequality} with a non-trivial constant $\lambda>0$, i.e
$$
\lambda\Var_{\mu_i}\,h\leq \Exp_{\mu_i}|h'|^2,\quad 1\leq i\leq n,\quad\mbox{ for every smooth function $h:\R\to\R$},
$$
then the product measure $\mu_1\times\dots\times \mu_n$ satisfies the Poincar\'e inequality \co{on} $\R^n$ with the same constant, which in turn implies subexponential dimension-free upper bound
$\exp(-c t)$ for \eqref{31941313049}, where $c>0$ depends only on the Poincar\'e constant \cite{GM83}
(see also, for example, \cite[Chapter~2]{RamonNotes}).
Conversely, if $\mu=\mu_1=\mu_2=\dots$ is a probability measure on $\R$, and for some $t>0$,
\eqref{31941313049}
is uniformly (over $n$) upper bounded by a quantity strictly less than $1/2$ then necessarily $\mu$ 
satisfies a Poincar\'e inequality with a non-trivial constant \cite{GRS2015}.

A connection between concentration and measure transport inequalities was first highlighted in \cite{Marton1,Marton2}.
In particular, it has been established in the literature (see \cite[Section~7]{OV}, \cite[Section~5]{Gozlan}, \cite[Corollary~5.1]{BGL})
that exponential dimension-free concentration for $\mu^{\times n}$, $n\geq 1$, is equivalent to the inequality
$$
\inf\limits_{X\sim \mu,\,Y\sim\nu} \Exp\,\min\big(|X-Y|,|X-Y|^2\big)
\leq C\int_\R \frac{d\nu}{d\mu}\log\big(\frac{d\nu}{d\mu}\big)\,d\mu $$
\co{ for every probability measure $\nu$ absolutely continuous w.r.t $\mu$, }
where the infimum is taken \co{ over all joint laws of $(X,Y)$ } with $X\sim\mu$ and $Y\sim\nu$.

A complete characterization of product measures which enjoy dimension-free {\it subgaussian} concentration
was obtained in \cite{Gozlan} (see also earlier work \cite{TalGaussT2}).
It was shown in \cite{Gozlan} that given a measure $\mu$ on $\R$, the quantity
in \eqref{31941313049} is upper bounded by $C\exp(-ct^2)$ for some $C,c>0$ (independent of $n$) if and only if
there is a constant $D>0$ such that $\mu$ satisfies the following measure transportation inequality (the {\it $T_2$-inequality}):
$$
\inf_{X\sim \mu,\;Y\sim \nu}\Exp\,|X-Y|^2\leq 2D\,\int_\R \frac{d\nu}{d\mu}\log\big(\frac{d\nu}{d\mu}\big)\,d\mu\quad
\mbox{for every probability measure $\nu$ absolutely continuous w.r.t $\mu$},
$$
where the infimum is over all pairs of random variables $X,Y$ on $\R$
with $X\sim\mu$ and $Y\sim\nu$.
We refer to \cite{Gozlan} for a more general statement.

We would like to mention the {\it logarithmic Sobolev inequality}
as a well known sufficient condition for subgaussian concentration \cite{DavSim}, \cite[Chapter~5]{Ledoux},
as well as inequalities interpolating between log-Sobolev
and Poincar\'e \cite{LO} as sufficient conditions for dimension-free concentration estimates of the form $\exp(-ct^p)$
for the quantities in \eqref{31941313049}.

\bigskip

Following works of Talagrand \cite{TalConvex1,TalConvex2}, it has been shown in various settings
that by restricting the class of Lipschitz functions to convex (or concave) functions, the worst-case
concentration estimates can be significantly improved.
As an illustration, it is well known that for every $n\geq 1$,
there exists a (non-convex) $1$-Lipschitz function $f_n$ in $\R^n$ such that for the random vector $X^{(n)}$
uniformly distributed on vertices of the cube $\{-1,1\}^n$, one has $\Var\,f_n(X^{(n)})=\theta(\sqrt{n})$
(see, for example, \cite[Problem~4.9]{RamonNotes}).
On the other hand, a classical result of Talagrand \cite{TalConvex1,TalConvex2} asserts that there is a universal constant $c>0$ such that,
with $\mathcal F_n:=\{\mbox{Convex $1$-Lipschitz functions in $\R^n$}\}$, and
with $\mu_1=\mu_2=\dots=\mu_n$ being the uniform measure on $\{-1,1\}$, the quantity in \eqref{31941313049}
is upper bounded by $\exp(-c\,t^2)$, for a universal constant $c>0$.
An extension of Talargand's argument shows that \eqref{31941313049} can be upper bounded by $\exp(-c\,t^2/d^2)$
for the class of convex $1$-Lipschitz
functions whenever $\mu_1,\dots,\mu_n$ are \co{measures with supports of maximum diameter at most $d>0$ \cite[Chapter~4]{Ledoux}:
\begin{equation}\label{apkejfnpfiunqpfijn}
\begin{split}
\sup\limits_{f\tiny\mbox{ convex $1$-Lipschitz}}\;&\max\big((\mu_1\times\dots\times \mu_n)\big\{f-\Med\, f\geq t\big\},
(\mu_1\times\dots\times \mu_n)\big\{f-\Med\, f\leq -t\big\}\big)\\
&\leq \exp\big(-c\,t^2/d^2\big),\quad t>0.
\end{split}
\end{equation}
A }complete characterization of probability measures $\mu$ on $\R$ such that
\eqref{31941313049} admits dimension-free subgaussian concentration for convex $1$-Lipschitz
functions with $\mu=\mu_1=\mu_2=\dots$,
was obtained in \cite{GRST2017,GRSST} (see also \cite{Adamczak} for an earlier result in this direction).
Both necessary and sufficient condition in that setting is $\mu((t+s,\infty))\leq 2\exp(-cs^2)\mu((t,\infty))$ and
$\mu((-\infty,-t-s))\leq 2\exp(-cs^2)\mu((-\infty,-t))$ for all $s,t>0$ for some constant $c>0$,
which can be interpreted as the condition that the distribution $\mu$ has ``no gaps''.
The convex subgaussian concentration, in turn, is implied by the {\it convex} log-Sobolev inequality (see \cite{SS}).
For results dealing with dimension-free subexponential-type concentration for convex Lipschitz functions, we refer
to \cite{BG1999,GRSST,AS2019,GRS2015}.

\bigskip

Whereas necessary and sufficient conditions for dimension-free concentration are well understood,
those conditions are rather strong.
For example, it is easy to construct an unbounded \co{subgaussian} distribution which does not satisfy the condition
for dimension-free subgaussian concentration mentioned above.

The main purpose of this note is to give {\it optimal
dimension-dependent} concentration bound in the class of {\it subgaussian product measures}
for convex $1$-Lipschitz functions.
However, we would like to start with a discussion of $\|\cdot\|_{\psi_p}$-bounded variables for $p\in(0,2)$, to emphasize the difference in tail behavior.
We recall the definition of the $\|\cdot\|_{\psi_p}$-(quasi-)norm. Given a real valued random variable $Y$, we set
$$
\|Y\|_{\psi_p}:=\inf\big\{\lambda>0:\,\Exp\,\exp(|Y|^p/\lambda^p)\leq 2\big\},\quad p>0.
$$
\co{In particular, $\|Y\|_{\psi_2}$ is the subgaussian constant of $Y$,
and $\|Y\|_{\psi_1}$ is the subexponential constant.
A random variable $Y$ with a bounded $\|\cdot\|_{\psi_p}$-norm satisfies, in view of Markov's inequality,}
$$
\Prob\{ |Y| \ge t \} \leq 2\exp( - t^p/\|Y\|_{\psi_p}^p),\quad t>0.
$$
\begin{theor}\label{9871-5091-09}
For every $p\in\co{(0},2)$ there is a $c_p>0$ depending only on $p$ with the following property. Let $K>0$, $n\geq 2$, and let
$X=(X_1,X_2,\dots,X_n)$ be a vector of independent random variables with $\|X_i\|_{\psi_p}\leq K$, $1\leq i\leq n$.
Then for every $1$-Lipschitz convex function $f$ in $\R^n$,
we have
$$
\Prob\big\{|f(X)-\Med\,f(X)|\geq t\big\}\leq 2\exp\big(-c_p\, t^p/K^p\big)
+2\exp\big(-c_p\,t^2/\big(K^2(\log n)^{2/p}\big)\big),\quad t>0.
$$
\end{theor}
We were not able to locate the above theorem in the literature, and provide its proof for completeness.
Theorem~\ref{9871-5091-09} is obtained by a simple reduction to Talagrand's
inequality for bounded variables. We note here that the two-level tail behavior for functions of independent
variables is a common phenomenon within high-dimensional probability, starting with the classical Bernstein's inequality.
It can be informally justified by saying that while deviation of individual variables from the above theorem
are controlled by $\exp(-\Theta(t^p))$, linear combinations of variables of the form $\sum_{i=1}^n a_i X_i$
(with $\|a\|_\infty\ll \|a\|_2$) exhibit subgaussian behavior in a certain range.
Notice that, in the above statement, $2\exp\big(- c_p \,t^2/\big(K^2(\log n)^{2/p}\big)$ is the
dominating term on the right hand side when $\frac{t}{K} = O\big( (\log n )^{\frac{2}{p(2-p)}}\big)$.
Further, there is no concentration phenomenon when $\frac{t}{K} = O( (\log n)^{1/p})$.
For $t\gg K(\log n )^{\frac{2}{p(2-p)}}$, the tail is estimated by $O\big(\exp\big(-c_p\, t^p/K^p\big)\big)$.

It can be verified that the statement of Theorem~\ref{9871-5091-09} is optimal in the following sense (we only consider the range $p\in[1,2)$ here).
\begin{prop}\label{39-109480980-98}
For every $p\in[1,2)$ there is a $C_p>0$ depending only on $p$ with the following property.
Let $n\geq C_p$, $t>0$, and $K>0$. Then there exist a random vector $X=(X_1,X_2,\dots,X_n)$
of independent random variables
with $\|X_i\|_{\psi_p}\leq K$, $1\leq i\leq n$, and a convex $1$-Lipschitz function $f$ such that
$$
\Prob\big\{ f(X)-\Med\,f(X) \geq t\big\} \geq \tilde c\max \big( \exp\big(- \tilde C \, t^2/\big(K^2(\log n)^{2/p}\big),
\exp\big(
-\tilde C\,t^p/K^p\big) \big)
$$
and
$$
 \Prob\big\{ f(X)-\Med\,f(X) \leq -t\big\} \geq \tilde c\max \big( \exp\big(- \tilde C \, t^2/\big(K^2(\log n)^{2/p}\big),
\exp\big( -\tilde C\,t^p/K^p\big) \big). 
$$ 
Here, $\tilde c,\tilde C>0$ are universal constants.
\end{prop}

\medskip

Now, let $X=(X_1,\dots,X_n)$ be a vector of independent $K$-subgaussian random variables, that is, $\|X_i\|_{\psi_2}\leq K$, $1\leq i\leq n$.
It is elementary to see that $\|X\|_\infty:=\max\limits_{i\leq n}|X_i|=O(\sqrt{\log n})$ with probability, say, $1-n^{-10}$,
where the implicit constant in $O(\cdot)$ depends on $K$. By considering the vector of truncations
$\big(X_i\,{\bf 1}_{\{|X_i|\leq C\sqrt{\log n}\}}\big)_{i=1}^n$ (for an appropriate choice of $C$) and applying the
Talagrand convex distance inequality, it is easy to deduce that for every $1$-Lipschitz convex function $f$ in $\R^n$,
$$
\Var\,f(X_1,\dots,X_n)=O(\log n),
$$
where the implicit constant depends on $K$ only. \co{A more elaborate argument \cite[Lemma~1.8]{KlochkovZ} gives, with the above notation, the following variable-dependent estimate:
$$
\Prob\big\{|f(X)-\Med\,f(X)|\geq t\big\}\leq 2\exp\bigg(-\frac{ct^2}{\big\|\max\limits_{1\leq i\leq n}|X_i|\big\|_{\psi_2}^2}\bigg),\quad t>0
$$
(see also \cite{MendelsonT}). When bounding the right hand side as a function of $n$, $K$, and $t$ only, we get
$$
\Prob\big\{|f(X)-\Med\,f(X)|\geq t\big\}\leq 2\exp\bigg(-\frac{ct^2}{K^2\log n}\bigg),\quad t>0,
$$
which is not sharp for large $t$ as our main result below shows.}
\begin{theor}\label{1-3974-1948798}
There is a universal constant $c>0$ with the following property. Let $K>0$, $n\geq 2$, and let
$X=(X_1,X_2,\dots,X_n)$ be a vector of independent $K$-subgaussian random variables.
Then for every $1$-Lipschitz convex function $f$ in $\R^n$,
we have
$$
\max\big(\Prob\big\{f(X)-\Med\,f(X)\geq t\big\},\Prob\big\{f(X)-\Med\,f(X)\leq -t\big\}\big)\leq \exp\bigg(
-\frac{c\,t^2}{K^2\log\big(2+\frac{K^2 n}{t^2}\big)}\bigg),\quad t>0.
$$
\end{theor}
The estimate provided by the theorem is optimal in the following sense:
\begin{prop} \label{prop: exampleSubgaussian}
Let $K>0$, $n\geq \tilde C$, and $t>0$. Then there exist a vector $X=(X_1,X_2,\dots,X_n)$
of independent $K$-subgaussian random variables, and a convex $1$-Lipschitz function $f$
such that
$$
\Prob\big\{f(X)-\Med\,f(X)\geq t\big\}\geq \tilde c\exp\bigg(
-\frac{\tilde C\,t^2}{K^2\log\big(2+\frac{K^2n}{t^2}\big)}\bigg),
$$
and
$$
\Prob\big\{f(X)-\Med\,f(X)\leq -t\big\}\geq \tilde c\exp\bigg(
-\frac{\tilde C\,t^2}{K^2\log\big(2+\frac{K^2n}{t^2}\big)}\bigg).
$$
Here, $\tilde c,\tilde C>0$ are universal constants.
\end{prop}

\medskip

The structure of the note is as follows.
In Section~\ref{-3194719847}, we provide a proof of Theorem~\ref{9871-5091-09}.
Section~\ref{1928471092847} is devoted to proving Propositions~\ref{39-109480980-98} and~\ref{prop: exampleSubgaussian}.
Finally, in Section~\ref{1876491874968} we consider the main result of the note, Theorem~\ref{1-3974-1948798}.

\section{Proof of Theorem~\ref{9871-5091-09}}\label{-3194719847}
Fix $p\in \co{(0},2)$, $K>0$, a natural number $n\geq 2$, and a $1$-Lipschitz convex function $f$ in $\R^n$.
To prove the theorem, it is sufficient to verify a deviation inequality for the parameter $t\geq CK(\log n)^{1/p}$,
where $C>0$ is a large constant depending on $p$.
Let $X=(X_1,\dots,X_n)$ be a vector of independent variables with $\|X_i\|_{\psi_p}\leq K$, $1\leq i\leq n$.

For each number $k\geq 1$, denote
$$
Y_i^{(k)}:=X_i\,{\bf 1}_{\{|X_i|\leq  \,\co{2^{k-1}}K({\co 4}\log n)^{1/p}\}}.
$$
Further, let $m\geq 1$ be the largest integer such that
\begin{equation}\label{efnqpifjnefpqifnpin}
\frac{t}{\co{2^{m}}\,K(\co{4}\log n)^{1/p}}\geq 1,
\end{equation}
and define
\begin{equation}\label{akfjnapfijnpfeinqpifnpjn}
u_k:=\tilde c\,2^{-(2-p)|m-k|/4},\quad k\geq 1,
\end{equation}
where the constant $\tilde c=\tilde c(p)>0$ is defined via the relation
$$
\tilde c\,\sum_{k=1}^\infty 2^{-(2-p)|m-k|/4}=\frac{1}{2}.
$$
We start by writing
\begin{align*}
\Prob&\big\{|f(X)-\Med\,f(X)|\geq t\big\}\\
&\leq \Prob\big\{|f(Y_1^{(1)},\dots,Y_n^{(1)})-\Med\,f(X)|\geq t/2\big\}
+\sum_{k=1}^\infty \Prob\big\{|f(Y_1^{(k+1)},\dots,Y_n^{(k+1)})-f(Y_1^{(k)},\dots,Y_n^{(k)})|\geq u_k\, t\big\}.
\end{align*}
To estimate the probability $\Prob\big\{|f(Y_1^{(1)},\dots,Y_n^{(1)})-\Med\,f(X)|\geq t/2\big\}$,
we note that the diameter of the support of each $Y_i^{(1)}$
is at most \co{ $2\,K(4\log n)^{1/p}$}, and hence
applying Talagrand's convex distance inequality for bounded variables \eqref{apkejfnpfiunqpfijn}, we get
$$
\Prob\big\{|f(Y_1^{(1)},\dots,Y_n^{(1)})-\Med\,f(Y_1^{(1)},\dots,Y_n^{(1)})|\geq s\big\}
\leq 2\exp\bigg(-\frac{c s^2}{K^2(\co{4}\log n)^{2/p}} \bigg),\quad s>0,
$$
for a universal constant $c>0$.
On the other hand,
we observe that
\begin{align*}
\min&\big(\Prob\big\{f(Y_1^{(1)},\dots,Y_n^{(1)})\geq \Med\,f(X)\big\},
\Prob\big\{f(Y_1^{(1)},\dots,Y_n^{(1)})\leq \Med\,f(X)\big\}\big)\\
&\geq \frac{1}{2}-n\,\max\limits_{i\leq n}\Prob\big\{|X_i|\geq \co{K(4}\log n)^{1/p}\big\}
\geq \frac{1}{2}-\frac{2}{n^3}\geq \frac{1}{4},
\end{align*}
which, together with the last inequality, implies that
$$
\big|\Med f(Y_1^{(1)},\dots,Y_n^{(1)})-\Med\,f(X)\big|\leq
CK(\co{4}\log n)^{1/p}.
$$
Therefore,
$$
\Prob\big\{|f(Y_1^{(1)},\dots,Y_n^{(1)})-\Med\,f(X)|\geq t/2\big\}
\leq 2\exp\bigg(-\frac{c t^2}{K^2(\co{4}\log n)^{2/p}} \bigg),
$$
for some universal constant $c>0$.

Further, for every $k\geq 1$ we have
\begin{align*}
\Prob&\big\{|f(Y_1^{(k+1)},\dots,Y_n^{(k+1)})-f(Y_1^{(k)},\dots,Y_n^{(k)})|\geq u_k\, t\big\}\\
&\leq \Prob\big\{\big\|\big(Y_i^{(k+1)}-Y_i^{(k)}\big)_{i=1}^n\big\|_2\geq u_k\, t\big\}\\
&\leq\Prob\bigg\{\sum_{i=1}^n {\bf 1}_{\{Y_i^{(k+1)}-Y_i^{(k)}\neq 0\}}\geq
\max\Big(1,\frac{u_k^2\, t^2}{\co{2^{2k}}\,K^2 ( \co{4}\log n)^{2/p}}\Big)\bigg\},
\end{align*}

where ${\bf 1}_{\{Y_i^{(k+1)}-Y_i^{(k)}\neq 0\}}$, $1\leq i\leq n$, are independent
Bernoulli random variables with
$$\Prob\big\{Y_i^{(k+1)}-Y_i^{(k)}\neq 0\big\}\leq \Prob\big\{|X_i|\geq
2^{k-1}K(4\log n)^{1/p}\big\}\leq \frac{2}{\co{ \exp(2^{kp-p}4\log n)}}
\leq \frac{1}{n^3},\quad 1\leq i\leq n.$$

A standard estimate $s^{\lceil \tilde s\rceil}{n\choose \lceil \tilde s\rceil}
\leq \big(\frac{en s}{\tilde s}\big)^{\tilde s}$ valid for any $\tilde s\in[1,n]$
and $s\in(0,(en)^{-1}]$, then implies
\begin{align*}
\Prob&\big\{|f(Y_1^{(k+1)},\dots,Y_n^{(k+1)})-f(Y_1^{(k)},\dots,Y_n^{(k)})|\geq u_k\, t\big\}\\
&\leq
\bigg(\frac{2en}{\exp(\co{2^{kp-p}\cdot 4}\log n)\max(1,\frac{u_k^2\, t^2}{\co{2^{2k}K^2(4\log n)^{2/p}}})
}\bigg)^{\max\big(1,\frac{u_k^2\, t^2}{\co{ 2^{2k}K^2(4\log n)^{2/p}}}\big)}\\
&\leq \exp\bigg(-c\,\co{ 2^{kp-p}} (\log n)\,\max\Big(1,\frac{u_k^2\, t^2}{\co{2^{2k}K^2(4\log n)^{2/p}}}\Big)\bigg)
\end{align*}
for some universal constant $c>0$,
\co{where the last inequality follows since $2en\leq \exp((2+\log_2(e))\log n)$}.

For $k\leq m$, we use the inequality 
\begin{align*}
\frac{t^2}{2^{2k}K^2 (4\log n)^{2/p}} \ge 2^{2m-2k},
\end{align*}
which follows \co{from \eqref{efnqpifjnefpqifnpin}, to write
\begin{align*}
\exp\bigg(-c\,2^{kp-p}(\log n)\,\max\Big(1,\frac{u_k^2\, t^2}{2^{2k}K^2(4\log n)^{2/p}}\Big)\bigg)
&\leq
\exp\big(-c(\log n)\,2^{kp-p +2m-2k}u_k^2\big)\\
&= \exp\big( -c (\log n) 2^{(m-1)p} 2^{(2-p)(m-k)}u_k^2\big)\\
&= \exp\big( -c \tilde c^2 2^{(m-1)p}\cdot (\log n)2^{(2-p)(m-k)/2}\big),
\end{align*}
where the last equality follows from \eqref{akfjnapfijnpfeinqpifnpjn}.
}
Using the definition of $m$ and assuming the constant $C$ in the assumption for $t$ is sufficiently large,
we get
\begin{align*}
\sum_{k\leq m}\Prob\big\{|f(Y_1^{(k+1)},\dots,Y_n^{(k+1)})-f(Y_1^{(k)},\dots,Y_n^{(k)})|\geq u_k\, t\big\}
&\leq\co{\sum_{k\leq m}
\exp\big(-c\,\tilde c^2\,\,2^{(m-1)p}\cdot (\log n)\;\,2^{(2-p)(m-k)/2}\big),}\\
&\leq\exp\big(-c'''\cdot (\log n)
\;\,2^{mp}\big)
\leq 
\exp\bigg(-\frac{\hat c\,t^p}{K^p}\bigg)
\end{align*}
for some $c''',\hat c>0$ depending only on $p$.

For $k>m$, we simply write
\begin{align*}
\Prob\big\{|f(Y_1^{(k+1)},\dots,Y_n^{(k+1)})-f(Y_1^{(k)},\dots,Y_n^{(k)})|\geq u_k\, t\big\}
\leq\exp\big(-c\,\co{2^{(m-1)p}\cdot 2^{(k-m)p}}\log n\big),
\end{align*}
and essentially repeating the above computations, get
\begin{align*}
\sum_{k>m}\Prob\big\{|f(Y_1^{(k+1)},\dots,Y_n^{(k+1)})-f(Y_1^{(k)},\dots,Y_n^{(k)})|\geq u_k\, t\big\}
\leq
\exp\bigg(-\frac{c''\,t^p}{K^p}\bigg)
\end{align*}
for some $c''>0$ depending only on $p$.

The result follows.
$\hfill \square $


\section{Proof of Propositions~\ref{39-109480980-98} and~\ref{prop: exampleSubgaussian}}\label{1928471092847}

First, consider the following basic example.
Let $p\in[1,2]$, $\tilde K>0$, and let $\mu$ be the probability measure on $\R$ defined via the relation
$$ \mu([t,\infty)) = \mu((-\infty, -t]) = \frac{1}{2}\exp\big(-(t/\tilde K)^p\big),\quad t\geq 0.$$
It is easy to see that, with the random vector $X$ in $\R^n$ distributed according to $\mu^{\times n}$,
the components of $X$ have $\|\cdot\|_{\psi_p}$-norms bounded by $O(\tilde K)$ (with the absolute implicit constant).
On the other hand, with the function $f:\R^n\to\R$ given by
$$
f(x_1,x_2,\dots,x_n):=x_1,\quad (x_1,x_2,\dots,x_n)\in\R^n,
$$
we have
$$
\Prob\{ f(X) \le -t \}= \Prob \{ f(X) \ge t \} = \frac{1}{2} \exp\big(-(t/\tilde K)^p\big),\, t >0,
$$
which gives the required estimates for $t\geq \tilde K(\log n )^{\frac{2}{p(2-p)}}$ in the statement of Proposition \ref{39-109480980-98},
and for $t \ge \tilde K\sqrt{n}$ in Proposition \ref{prop: exampleSubgaussian}.

\medskip

The main statement of this section is the following proposition.
 \begin{prop} \label{prop: example}
    There exists a universal constant $C>1$ so that the following holds:
    Let $n\geq C$, $p \in [1,2]$, $K>0$.  
    Further, let $0\le t \le \frac{K}{C}\sqrt{n}$. Then there exists a random vector
    $X=(X_1,\dots,X_n)$ with i.i.d components whose $\|\cdot\|_{\psi_p}$-norm is bounded above by $K$ such that 
   \begin{align*} 
   \Prob\big\{\|X\|_2-\Med\,\|X\|_2\geq t\big\}&\geq \frac{1}{C} \exp\bigg(
- C\cdot \frac{t^2}{K^2 \big(\log\big(2+\frac{ K^2n }{ t^2 }\big) \big)^{2/p} }\bigg), \quad\mbox{ and } \\
 \Prob\big\{\|X\|_2-\Med\,\|X\|_2\leq -t\big\} &\geq \frac{1}{C} \exp\bigg(
- C\cdot \frac{t^2}{K^2 \big(\log\big(2+\frac{K^2 n}{t^2 }\big) \big)^{2/p} }\bigg).
    \end{align*}
 \end{prop}
 Together with the above example, Proposition~\ref{prop: example} implies 
 Propositions~\ref{39-109480980-98} and ~\ref{prop: exampleSubgaussian}.
The ``test'' distribution we use to prove Proposition~\ref{prop: example}
is the $n$-fold product of a $2$-point probability measure defined by $\mu(\{0\})=1-\theta$ and $ \mu( \{K\log(1/\theta)^{1/p}\} ) = \theta$ where $\theta=\theta(t)$ is an appropriately chosen parameter. 

\medskip

The proof of the proposition relies on a precise lower bound for the tail probability of a Binomial random variable. We need the following result: 
 \begin{lemma} \label{prop: BinomialTail} 
  There exists universal constants \co{$c_{b}\in(0,1)$} and $C_{b}>1$ so that the following holds. 
  Let $n$ be a sufficiently large integer. 
  For $\theta\in\left[\frac{1}{c_{b}\,n},c_{b}\right]$, let $Y_{1},\dots,Y_{n}$
  be i.i.d Bernoulli random variables with a parameter $\theta>0$. Then, for any $0\le r\le n-\theta n$, 
 we have
  \begin{align} \label{eq: BinomialTail}  
 \mathbb{P}\Big\{ \sum_{i=1}^{n}Y_{i}\ge \theta n+r\Big\} 
 \ge \frac{1}{C_{b}}\exp\left(- C_{b}\log\left(2+\frac{\theta n+r}{\theta n}\right)\frac{r^{2}}{\theta n+r}\right).
  \end{align}
  \end{lemma}
    \begin{rem}
  The term $\frac{r^{2}}{\theta n+r}$ corresponds to the usual Bernstein-type
  tail estimate, and $\log\big(2+\frac{\theta n+r}{\theta n}\big)$ is the ``extra'' factor
  emerging when $\theta n=o\left(r\right)$.
  \end{rem}
Although the above statement is based on completely standard calculations, we provide its proof for completeness.
    \begin{proof}[Proof of Lemma~\ref{prop: BinomialTail}]
  We will assume that $ \sqrt{\theta n}$ (and  $\theta n$) is greater 
  than a sufficiently large universal constant and at the same time $ \theta$ is
  smaller than another small universal constant. Those conditions on $\theta$ can 
  be imposed by adjusting the constant $c_{b}$ in the statement of the lemma. 
  For every $k\leq n$, let $P_{k}:=\mathbb{P}\left\{ \sum_{i=1}^{n}Y_{i}=k\right\} $ and
  $P_{\ge k}:=\mathbb{P}\left\{ \sum_{i=1}^{n}Y_{i}\ge k\right\} $.

  We claim that in order to prove the lemma it is sufficient to establish the following inequalities:
  \begin{align} \label{eq: sufficientTail}
      \forall \; 0 \le r \le n- \theta n
      \mbox{ with }  \theta n + r \in \mathbb{N},
      \quad
      P_{\geq \theta n + r} \ge \begin{cases}
        \frac{1}{\bar{C}}
        \exp \Big( - \bar{C} 
            \log \Big( 2+ \frac{r}{\theta n}\Big)
        r \Big) &  \mbox{ if } r \ge \frac{1}{10} \theta n, \\
        \frac{1}{\tilde{C}}
        \exp \Big(  -\tilde{C} 
        \frac{r^2}{\theta n +r}
          \Big)  & \mbox{ if } 0 \le r < \frac{1}{10} \theta n,
      \end{cases}
  \end{align}
  for some universal constants $ \tilde{C}, \bar{C} >1$.
  
  To verify the claim, fix any $\theta$ (satisfying assumptions from the beginning of the proof)
  and any $r$ with $0< r\leq n- \theta n$.
  We have 
  $P_{\ge \theta n +r} = P_{\ge \lceil \theta n + r \rceil }$. 
  
  First, consider the case $\lceil \theta n + r \rceil-\theta n \ge \frac{1}{10}\theta n$.
  Since $\theta n$ is greater than a large universal constant, we have
  $ \lceil \theta n + r\rceil \le \theta n + 2r$, whence, applying \eqref{eq: sufficientTail} with parameters
  $\theta$ and $\lceil \theta n + r \rceil-\theta n$,
  $$ P_{ \ge \lceil \theta n + r\rceil } \ge 
    \frac{1}{\bar{C}} 
    \exp\Big( - \bar{C}\log\Big(2+ \frac{2r}{\theta n} \Big) 
    \cdot 2r \Big) 
  \ge 
  \frac{1}{\bar{C}} \exp\Big( 
    - 4\bar{C}\log\Big(2+ \frac{r}{\theta n}\Big) r
  \Big),$$
  where the last inequality holds since $ \log(2+2x) \le \log \big( (2+x)^2 \big)
  = 2\log(2+x)$ for $x\ge 0$. Further, under the condition $\lceil \theta n + r \rceil-\theta n \ge \frac{\theta n}{10}$ and assuming that $\theta n$ is larger than a big universal constant,  
  we have $ \frac{ 12 r}{\theta n +r} \ge 1$. Therefore, 
  \begin{align*} 
    P_{ \ge \lceil \theta n + r\rceil } \ge 
    \frac{1}{\bar{C}} 
    \exp\Big( - 4\bar{C}\log\Big(2+ \frac{r}{\theta n} \Big) 
    r \Big) 
  \ge 
  \frac{1}{\bar{C}} \exp\Big( 
    - 4\cdot 12\bar{C}\log\Big( 2+\frac{\theta n+r}{\theta n}\Big)\frac{r^2}{\theta n +r}\Big) .
  \end{align*}
  
  Next, consider the case $ 0 < r$, $\lceil \theta n + r \rceil-\theta n<\frac{\theta n}{10}$. Clearly, $ \lceil \theta n + r \rceil - \theta n \le r+1$, and hence
   $$ P_{ \ge \lceil \theta n + r\rceil } \ge 
    \frac{1}{\tilde{C}} 
    \exp\Big( - \tilde{C} \frac{(r+1)^2}{\theta n + r}\Big) 
    \ge
    \frac{1}{\tilde{C}} 
     \exp\Big(  - \tilde{C} \frac{r^2}{\theta n +r} - \tilde{C}        
     \Big),
  $$
  where the last inequality holds since $ r \le \frac{ \theta n }{10}$ and $\theta n$
  is sufficiently large. As $ \log \Big( 2 + \frac{r}{\theta n}\Big) \ge \log(2)$, we obtain
  \begin{align*} 
      P_{ \ge \lceil \theta n + r\rceil }
  \ge 
    \frac{1}{\tilde{C}} \exp(-\tilde{C})
    \exp\Big( - \frac{\tilde{C}}{\log(2)} 
      \log \Big(2+ \frac{r}{\theta n}\Big) 
     \frac{r^2}{\theta n + r}\Big),
  \end{align*}
  and derivation of \eqref{eq: BinomialTail} from \eqref{eq: sufficientTail} is complete.

\medskip


  From now on, we assume \co{$0 \le r \le n-\theta n$} and $ \theta n + r \in \mathbb{N}$. 
  Obviously,
  \begin{align} \label{eq: PthetanrFormula}
  P_{\theta n+r}={n \choose \theta n+r} \theta ^{\theta n+r}
  \left(1-\theta \right)^{n-\theta n-r}. 
  \end{align}

  {\bf Case 1:} $ \frac{ \theta n }{10} \le r \le n-\theta n$.
  
  By the standard estimate,
  $  { n \choose \theta n+r } \ge  \big(\frac{n}{\theta n + r}\big)^{\theta n +r}$, and so
  \[P_{\theta n+r}  \ge  \bigg(\frac{\theta n}{\theta n + r}\bigg)^{\theta n+r} ( 1 - \theta)^{n -\theta n -r}
  = \exp \bigg( - \log \bigg( \frac{\theta n + r}{\theta n} \bigg) (\theta n + r) \bigg)  (1-\theta)^{n -\theta n - r}.
 \]
 Since $(1-\theta) \ge \exp(-2\theta)$ whenever $\theta>0$ is small enough, we get
  \[
    (1-\theta)^{n-\theta n -r} \ge (1-\theta)^n \ge \exp(-2\theta n),
  \]
  and therefore 
  \[
    P_{\ge \theta n +r} \ge 
    P_{\theta n + r} \ge 
    \exp \bigg( - \log \bigg( \frac{\theta n + r}{\theta n} \bigg) (\theta n + r) -2 \theta n \bigg)  
    \ge \exp \bigg( - C \log \bigg( \frac{\theta n + r}{\theta n} \bigg) r \bigg)  
  \]
  for a universal constant $C>1$.   
  This completes the proof of \eqref{eq: sufficientTail} in the regime $ r \ge \frac{\theta n}{10}$. 
  
  {\bf Case 2: } $ 0 \le r < \frac{\theta n}{10}$.
  
  In view of Stirling's formula, 
  \begin{align*}
  P_{ \theta n+r}\ge & c\,\sqrt{\frac{n}{\left(\theta n+r\right)\left(n-\theta n-r\right)}}
  \left(\frac{ n}{\theta n+r}\right)^{\theta n+r}\left(\frac{n}{n-\theta n-r}\right)^{n-\theta n-r}
  \theta^{\theta n +r} (1-\theta)^{n-\theta n - r} \\
  \ge & 
  \frac{ c}{\sqrt{\theta n+r}}
  \left(\frac{ \theta n}{\theta n+r}\right)^{\theta n+r}\left(\frac{n-\theta n}{n-\theta n-r}\right)^{n-\theta n-r},
  \end{align*} 
  where $c>0$ is a universal constant.
  Since $\log\left(1+x\right)\ge x-x^2$ for $x>0$, we get
  \[
  \left(\frac{n- \theta n}{n- \theta n-r}\right)^{n- \theta n-r}=\left(1+\frac{r}{n- \theta n-r}\right)^{n- \theta n-r}\ge\exp\left(r-\frac{r^{2}}{n-n \theta -r}\right).
  \]
  Similarly, since $\log\left(1-x\right)\ge-x- 2x^2 $ for $x\in[0, \frac{1}{2} ]$
  and $ \frac{r}{\theta n +r} \in [0, \frac{1}{2}]$ for $ 0\le r \le \frac{1}{10}\theta n$,
  \[
  \left(\frac{ \theta n}{ \theta n+r}\right)^{ \theta n+r}=\left(1-\frac{r}{ \theta n+r}\right)^{ \theta n+r}\ge\exp\left(-r-\frac{ 2r^{2}}{\theta n+r }\right).
  \]
  Hence, together using that  $  \frac{1}{n-\theta n -r} \le \frac{1}{n-\frac{11}{10}\theta n} \le \frac{1}{ \frac{11}{10}\theta n} \le \frac{1}{\theta n + r}$
  when $0 < \theta < \frac{1}{3}$, we get
  \begin{align} \label{eq: PthetanrSmall} 
  P_{\theta n+r}\ge\frac{c}{\sqrt{\theta n+r}}\exp\left(-\frac{ 3 r^{2}}{\theta n+r}\right).
  \end{align}
  The bound $ P_{ \ge \theta n +r} \ge P_{\theta n +r}$
  is insufficient to get $\eqref{eq: sufficientTail}$ when 
  $r$ is small. We will bound $P_{\ge \theta n +r}$ 
  by comparing it with the sum of a geometric sequence starting with $ P_{\theta n +r}$.

  For $r'>0$ with $ \theta n + r' \in \mathbb{N}$ and $n-\theta n -r'>0$, 
  by \eqref{eq: PthetanrFormula} we have 
  \begin{align*} 
  \frac{P_{\theta n+r'+1}}{ P_{\theta n+r'}}
  = \frac{n-\theta n -r'}{ \theta n + r'+1} \frac{\theta}{1-\theta}
  = \frac{ 1-\frac{r'}{\left(1-\theta \right)n}}{1+\frac{1+r'}{\theta n} }.
  \end{align*}
  Since $ \frac{1}{1+x} \ge 1- x$ for all $x\ge 0$, 
  $$
  \frac{P_{\theta n+r'+1}}{ P_{\theta n+r'}}
  \ge \bigg( 1- \frac{r'}{(1-\theta)n} \bigg)
    \bigg( 1- \frac{1+r'}{\theta n}\bigg)
  \ge 1- \frac{r'}{(1-\theta)n} - \frac{1+r'}{\theta n}.
  $$
  Next, with $\frac{\theta }{1-\theta} \le \frac{c_{b}}{1-c_{b}}  \le \frac{1}{3}$ 
  when $c_{b}>0$ is small enough, 
  \begin{align*} 
  \frac{ P_{\theta n+r'+1}}{ P_{\theta n+r'}}
  \ge  1- \frac{1+\frac{4}{3}r'}{\theta n}.
  \end{align*}
  Notice that for $0\leq i \le \max(\lceil r \rceil,\, \lceil \sqrt{\theta n} \rceil ) := u$,  we have
  $ 1-  \frac{   1+ \frac{4}{3}(r+i) }{\theta n} 
  \ge 1- \frac{ 4 u }{\theta n}
  $
  where we used that $ \sqrt{\theta n}$ is greater than a large
  absolute constant. 
   Hence, for $1\leq i \le u$, 
  $$
    P_{\theta n + r + i} 
  \ge  
    P_{\theta n + r} \Big( 1- \frac{4u}{\theta n}\Big)^{i}. 
  $$
  Then, 
  \begin{align*}
  P_{\ge \theta n+r}
  \ge 
  \sum_{i=0}^{u} P_{\theta n +r + i} 
  \ge 
  P_{\theta n+r} \cdot \bigg(\sum_{i=0}^{u}
 \Big(1-\frac{4u}{\theta n}\Big)^{i}\bigg)
  =
    P_{\theta n+r} \cdot \frac{1-\big(1-\frac{4u}{\theta n}\big)^{u+1}}{\frac{4u}{\theta n}}
  \ge P_{\theta n+r} \cdot \frac{\theta n}{8 u}
  \end{align*}
  where the last inequality holds since
 $ \big(1-\frac{4u}{\theta n}\big)^{u+1} 
 \le \exp \big( -\frac{4u^2}{\theta n} \big) 
 \le \exp(-4) \le \frac{1}{2}$ since $u \ge \sqrt{\theta n}$.
   Together with \eqref{eq: PthetanrSmall}, we obtain 
  \[
  P_{\ge \theta n+r} \ge  \frac{\theta n}{8u} \frac{c}{\sqrt{\theta n +r}} \exp\bigg( - \frac{3r^2}{\theta n + r} \bigg).
  \]
  With $ \theta n \ge \frac{\theta n + r}{2} $ (since $r \le \frac{\theta n}{10}$) and  
  $u \le 2 \max(r,\, \sqrt{\theta n} )$ (if $\theta n$ is large enough), 
  $\frac{\theta n}{8u} \frac{c}{\sqrt{\theta n +r}}
  \ge \frac{c}{32} \frac{\sqrt{\theta n +r}}{\max(r,\, \sqrt{\theta n})}
  $. Finally, it is easy to check that 
  $$
    \frac{\sqrt{\theta n +r}}{\max(r,\, \sqrt{\theta n})} \ge \exp\Big( - \frac{r^2}{\theta n +r}\Big).
  $$
  Now we conclude that 
  $$
    P_{\ge \theta n +r} \ge \frac{c}{32} \exp\bigg( - \frac{ 4r^2}{\theta n +r}\bigg),
  $$
  and the proof of \eqref{eq: sufficientTail} is finished. 
  \end{proof}

 \begin{lemma} \label{lem: UptailNorm}
    There exist constants $c_b>0$ and $\tilde{C}_b>1$ so that the following holds.
    Let $n$ be a sufficiently large integer and let $\alpha >0$.  
    For $ \theta \in \Big[ \frac{1}{c_{b}n}, c_{b} \Big]$, 
    let $Y_1, \dots, Y_n$ be i.i.d Bernoulli random variables with 
    parameter $\theta$. Set $X=(X_1,X_2, \dots, X_n)$, with 
    $X_i = \alpha Y_i$, $i\leq n$. Then, 
    for all $t \in \left[ 0, \frac{\alpha \sqrt{n}}{4} \right]$, 
    \begin{align} \label{eq: UptailNorm}
     \mathbb{P} \big\{ \|X\|_2 \ge 
        \Med\|X\|_2 + t \big\} \ge  \frac{1}{\tilde{C}_b}\exp\bigg( - \tilde{C}_b 
    \log\bigg(2 + \frac{t^2}{\theta n\alpha^2}\bigg)\frac{t^2}{\alpha^2} \bigg). 
    \end{align}
 \end{lemma} 
\begin{proof}
  Clearly,
  $$\|X\|_2= \alpha \sqrt{\sum_{i=1}^n Y_i}.$$ 
  Since the mapping $ y \mapsto \alpha \sqrt{y}$ is monotone increasing for $y \ge 0$, 
  the median estimate for Binomial random variable $$ \lfloor \theta n \rfloor \le  \Med\Big( \sum_{i=1}^n Y_i\Big) \le \lceil \theta n \rceil$$
  (see \cite{KB80}) implies 
  \begin{align} \label{eq: meanMedian}  
   \alpha \sqrt{ \lfloor\theta n \rfloor } \le \Med \|X\|_2  \le \alpha \sqrt{ \lceil \theta n \rceil }.  
   \end{align}
   Thus, $$ | \Med \|X\|_2 - \alpha \sqrt{\theta n} | \le 
    \alpha \sqrt{ \lceil \theta n \rceil } - \alpha \sqrt{ \lfloor \theta n \rfloor } \le \alpha, $$
    where the last inequality holds when $ \theta n \ge 1$. 
    
    We claim that in order to verify the lemma, it is sufficient to establish the following bound:
    \begin{align} \label{eq: sufficinetUptailNorm}
      \forall t \in \bigg[0, \alpha \frac{\sqrt{n}}{2}\bigg],\,\,\,
     \mathbb{P} \big\{ \|X\|_2 \ge 
         \alpha \sqrt{\theta n} + t \big\} \ge \frac{1}{C} \exp\bigg( - C \log\bigg(2 + \frac{t^2}{\theta n\alpha^2}\bigg)\frac{t^2}{\alpha^2} \bigg)
    \end{align}
    for a universal constant $C>1$. Indeed, suppose \eqref{eq: sufficinetUptailNorm} holds. For $t \in [0, \frac{\alpha \sqrt{n}}{4}]$,
    $$
      \Prob \big\{ \|X\|_2 \ge \Med \|X\|_2 + t \big\}
    \ge 
      \Prob \big\{ \|X\|_2 \ge \alpha \sqrt{ \theta n} + \alpha + t \}
    \ge \frac{1}{C} 
      \exp\bigg( - C \log\bigg(2 +  \frac{1}{\theta n} \bigg( \frac{t}{\alpha}+1 \bigg)^2\bigg)\bigg(\frac{t}{\alpha}+ 1\bigg)^2 \bigg),
    $$
    where the last inequality follows from \eqref{eq: sufficinetUptailNorm} since 
    $ \alpha +t \in \Big[ 0, \alpha \frac{\sqrt{n}}{2}\Big]$, under the assumption $ n \ge 16$. 
    Since $ (\frac{t}{\alpha}+1 )^2 \le 2(\frac{t}{\alpha})^2 +2$, we get
    $$ \log\bigg( 2 + \frac{1}{\theta n} \bigg( \frac{t}{\alpha}+1 \bigg)^2 \bigg) 
     \le \log \bigg( 2 + \frac{2}{\theta n} +  \frac{2}{\theta n} \bigg(\frac{t}{\alpha} \bigg)^2 \bigg) 
     \le \log \bigg( 2 \cdot \bigg( 2+ \frac{t^2}{\theta n \alpha^2} \bigg) \bigg)
     \le  2 \log \bigg( 2+ \frac{t^2}{\theta n \alpha^2}\bigg) ,
     $$
    where we used $ \frac{1}{\theta n} \le 1$ in the second inequality. 
    Then, applying the bounds $ (\frac{t}{\alpha}+1 )^2 \le 2(\frac{t}{\alpha})^2 +2$ and $ \frac{1}{\theta n} \le 1$ again, 
    we obtain
    $$
      \log\bigg(2 +  \frac{1}{\theta n} \bigg( \frac{t}{\alpha}+1 \bigg)^2\bigg)\bigg(\frac{t}{\alpha}+ 1\bigg)^2
    \le 
      4 \log\bigg(2 +  \frac{t^2}{\theta n \alpha^2 } \bigg) 
      \bigg(\bigg(\frac{t}{\alpha}\bigg)^2+1 \bigg) 
    \le 
      4 \log(3) + 8 \log\bigg(2 +  \frac{t^2}{\theta n \alpha^2 } \bigg) 
      \bigg(\frac{t}{\alpha}\bigg)^2,
    $$
    where we applied the inequality $\log\Big(2 +  \frac{t^2}{\theta n \alpha^2 } \Big) 
    \le \max \Big( \log(3), \log\Big(2 +  \frac{t^2}{\theta n \alpha^2 } \Big) \Big(\frac{t}{\alpha}\Big)^2\Big)$. Therefore, 
    $$
      \Prob \big\{ \|X\|_2 \ge \Med \|X\|_2 + t \big\}
      \ge 
       \frac{1}{C} \exp(-4 \log(3)C) \exp\bigg(- 8C \log\bigg(2 +  \frac{t^2}{\theta n \alpha^2 } \bigg) 
      \bigg(\frac{t}{\alpha}\bigg)^2 \bigg),
    $$ and
    \eqref{eq: UptailNorm} follows from \eqref{eq: sufficinetUptailNorm} with $C_b = \max( C\exp(4\log(3)C), 8C)$. 
The claim is established.

\medskip

  Now we prove \eqref{eq: sufficinetUptailNorm}. First, since $\|X\|_2= \alpha \sqrt{\sum_{i=1}^n Y_i}$, 
  \begin{align*}
    \mathbb{P} \big\{ \|X\|_2 \ge \alpha \sqrt{ \theta n} + t \big\}
  = &  \mathbb{P} \bigg\{ \sum_{i=1}^n Y_i - \theta n \ge 
    \underbrace{2 \sqrt{ \theta n}  \frac{t}{ \alpha } +  \frac{t^2}{\alpha^2} }_{r}
  \bigg\}.
  \end{align*}
  For $0 \le  \frac{t}{\alpha} \le \sqrt{ \theta n}$, we have
  $0 \le r  \le 3\theta n$. 
  We apply Lemma~\ref{prop: BinomialTail} and use that $ \log( 2 + \frac{\theta n+r}{\theta n}) \le \log(6)$, to conclude
  \begin{align*}
    \mathbb{P} \big\{ \|X\|_2 \ge \sqrt{\alpha \theta n} + t \big\}
    &\ge \frac{1}{C_b}\exp\bigg( - C_b\log(6)\cdot \frac{r^2}{\theta n}\bigg)\\
    &\ge \frac{1}{C_b}\exp\bigg( - C_b\log(6)\cdot 9  \frac{t^2}{\alpha^2} \bigg)
    \ge \frac{1}{C_b}\exp\bigg( - C_b\frac{9\log(6)}{\log(2)} \log\bigg(2 + \frac{t^2}{\theta n \alpha^2} \bigg)   \frac{t^2}{\alpha^2} \bigg).
  \end{align*}
   For $ \sqrt{\theta n} \le \frac{t}{\alpha} \le \frac{1}{2}\sqrt{ n}$, we have
   $ \theta n \le r \le  \frac{ 3t^2}{\alpha^2} \le \frac{3}{4}n \le n - \theta n$
   where the last inequality holds when $c_{b}>0$ is chosen small enough.   
   Applying Lemma~\ref{prop: BinomialTail} again, we obtain 
   \begin{align*}
    \mathbb{P} \big\{ \|X\|_2 \ge \alpha\sqrt{ \theta n} + t \big\}
    \ge& \frac{1}{C_b}\exp\bigg( - C_b \log\bigg(2 + \frac{ 6t^2}{\theta n \alpha^2 }\bigg) \cdot \frac{3t^2}{\alpha^2} \bigg).
   \end{align*}
  We have $\log\Big(2 + \frac{6t^2}{\theta n \alpha^2 }\Big)
   \le 3\log\Big(2 + \frac{t^2}{\theta n \alpha^2}\Big)$,
   and hence
\begin{align*}
    \mathbb{P} \big\{ \|X\|_2 \ge \alpha\sqrt{ \theta n} + t \big\}
    \ge& \frac{1}{C_b}\exp\bigg( - 9C_b \log\bigg(
    2 + \frac{t^2}{\theta n\alpha^2}  \bigg) \frac{t^2}{\alpha^2} \bigg).
   \end{align*}
   Now \eqref{eq: sufficinetUptailNorm} follows by choosing $C: = \max\big( \frac{9\log(6)}{\log(2)}, 9\big)\,C_b$.
   \end{proof}

 \begin{proof}[Proof of Proposition~\ref{prop: example}]
  Let $X(\theta)=(X_1(\theta),\dots, X_n(\theta))$ be the random vector defined in Lemma~\ref{lem: UptailNorm} with parameters $\theta \in [ \frac{1}{c_{b}n}, c_{b}]$ and $\alpha: = K(\log(1/\theta))^{1/p}$ (the actual choice of $\theta$ will be made later in the proof). 
  Then, $\{X_i(\theta)\}_{i=1}^n$ are i.i.d random variables with the $\|\cdot\|_{\psi_p}$-norm bounded above by $K$. 
  We want to emphasize that the distribution of $X$ depends on the parameter $\theta$, and 
  that our future choice of $\theta$ will also depend on $t$. 

  Applying Lemma~\ref{lem: UptailNorm} with $ 0\le t  \le \frac{K\sqrt{n}}{4} 
  \le \frac{\alpha \sqrt{n}}{4}$ and any $ \theta \in [ \frac{1}{c_{b}n}, c_{b}]$, we get
  \begin{align}\label{3124-18-0918-098-098}
   \mathbb{P} \big\{ \|X(\theta)\|_2 \ge \Med\|X( \theta)\|_2 +t \big\}
  \ge \frac{1}{\tilde{C}_b} \exp\bigg( - \tilde{C}_b\log \bigg( 2 + \frac{t^2}{K^2\theta n (\log(1/\theta))^{2/p} }\bigg) \frac{t^2}{K^2(\log(1/\theta))^{2/p} }\bigg). 
  \end{align}
  
  \medskip
  
  {\bf Case 1:} $t \in \Big[
  \sqrt{\frac{K^{2} (\log n)^{2/p}}{3c_{b}}}, \sqrt{\frac{c_{b}K^{2}n}{3}} \Big]$. In this case, we define
  \begin{align*}
    \theta := \theta(t) = \bigg(\frac{K^2n}{3t^2} \Big(
    \log\Big( \frac{K^2n}{3t^2} \Big) \Big)^{2/p} \bigg)^{-1}.
  \end{align*}
  Since $ t \mapsto \theta(t)$
  is a monotone increasing function for $ t \le K \sqrt{n/3}$, our choice of $\theta$ satisfies 
  \[
    \frac{1}{c_{b}n} \le 
    \underbrace{\frac{(\log n )^{2/p} }{c_{b}n \Big(\log\Big( \frac{ c_{b}n}{ (\log n )^{2/p} } \Big) \Big)^{2/p} }}_{{\rm when }\, t= \sqrt{\frac{K^{2}(\log n )^{2/p}}{3c_{b}}}}
    \le \theta \le \underbrace{\frac{c_{b}}{ \Big( \log \Big( \frac{1}{c_{b}}\Big)\Big)^{2/p}} }_{{\rm when}\, t=\sqrt{\frac{c_{b}K^{2}n}{3}} }  \le c_{b},
  \]
  which conforms to the conditions in Lemma~\ref{lem: UptailNorm}, and therefore the estimate~\eqref{3124-18-0918-098-098} is valid.
 Our choice of $\theta$ implies  $  \log( 1/\theta)  \ge \log\big( \frac{K^2 n}{3t^2} \big) $ and thus 
  \begin{align*}
   \log\left(2+\frac{3t^{2}}{K^{2}\theta n
   (\log\left(1/\theta\right))^{2/p}
   }\right)\frac{3t^{2}}{K^{2}
    (\log\left(1/\theta\right))^{2/p}} 
   =&  \log\bigg( 2 + 
    \frac{\big(\log\big(\frac{K^2n}{3t^2}\big)\big)^{2/p}}{(\log(1/\theta))^{2/p}} 
    \bigg)
    \frac{3t^2}{K^2(\log(1/\theta))^{2/p}}  \\
   \le  &
    \frac{3\log(3)\,t^2}{K^2\big( \log\big( \frac{K^2n}{3t^2}\big)\big)^{2/p}}.
  \end{align*}
 Further, the assumption that $ t \le \sqrt{\frac{c_b K^2n }{3}}$ and $c_b>0$ is sufficiently small implies that $ \frac{K^2n}{t^2} \ge 9$ and therefore 
 \begin{align} \label{eq: annoyingLog}
 \log\Big( \frac{K^2n}{3t^2}\Big) \ge  \frac{1}{2}\log\Big( \frac{K^2n}{t^2}\Big) = \frac{1}{4} \log\Big(\Big( \frac{K^2n}{t^2}\Big)^2\Big) 
    \ge \frac{1}{4}\log \Big( 2+ \frac{K^2n}{t^2}\Big).
 \end{align}
 We conclude that 
  \begin{align*}
  \mathbb{P}\big\{ 
    \|X(\theta(t))\|_2\ge \Med \|X(\theta(t))\|_2+t
    \big\} 
    &\geq
    \frac{1}{\tilde{C}_b} \exp\bigg( - \tilde{C}_b\frac{3\log(3)\,t^2}{K^2\big( \log\big( \frac{K^2n}{3t^2}\big)\big)^{2/p}}\bigg)\\
    &\ge  \frac{1}{\tilde{C}_b}\exp\bigg (- 3\cdot 4^{2/p}\tilde{C}_b \log(3)\, \frac{t^{2}}{K^{2}\big(\log\big (2+\frac{ K^2 n}{ t^2 }\big)\big)^{2/p}}\bigg).
  \end{align*}
   Next, we will handle the lower tail estimate. 
   We can assume that $ \lfloor \theta n \rfloor \ge \theta n/3$ since $\theta n \ge  \frac{1}{c_{b}}$ and
   $c_{b}>0$ is sufficiently small. Then, by \eqref{eq: meanMedian} we have
  \begin{align*}
  \Med\|X(\theta(t))\|_2 &\ge  K(\log( 1/\theta))^{1/p} \sqrt{\lfloor\theta n \rfloor }
  \ge
  K(\log( 1/\theta))^{1/p} \sqrt{\theta n/3 }\\
   &= 
   \sqrt{ \bigg(\log\Big(\frac{K^2n}{3t^2} \Big(\log \Big( \frac{K^2n}{3t^2} \Big)\Big)^{2/p}\Big)\bigg)^{2/p}
   \frac{t^2}{(\log(K^2n/3t^2))^{2/p}} }   
   \ge t. 
  \end{align*}
  As a consequence, 
  \begin{align*}
  \mathbb{P} \big\{ \|X(\theta(t))\|_2\le \Med\|X(\theta(t))\|_2-t \big\} 
  &\ge  \mathbb{P}\left\{ \|X(\theta(t))\|_2 = 0 \right\} = (1-\theta)^n \ge \exp\big( -2\theta n \big)\\
  &= \exp\bigg( - \frac{6t^2}{K^2 \big( \log(K^2n/3t^2) \big)^{2/p} }\bigg).
  \end{align*}
  Finally, by \eqref{eq: annoyingLog}, 
   \[
  \mathbb{P}\big\{ \|X(\theta(t))\|_2\le \Med \|X(\theta(t))\|_2-t\big\} 
  \ge  \exp\bigg(-6\cdot 4^{2/p}\frac{t^{2}}{K^{2} \big( \log\big(
  2+\frac{K^2 n}{ t^2 }
  \big)\big)^{2/p} }\bigg).
  \]
We have shown that for $
  \sqrt{\frac{K^{2} (\log n)^{2/p}}{3c_{b}}}
 \le t\le \sqrt{\frac{c_{b}K^{2}n}{3}}$, the proposition holds 
 with $C=  \max(48\tilde{C}_b \log(3),\, 6\cdot 16)$,  since $ p \ge 1$. 

 {\bf Case 2: } $ 0\le t  \le \sqrt{\frac{K^{2} (\log n)^{2/p}}{3c_{b}}}$.
  Set $ t_0: = \sqrt{ \frac{K^2 ( \log n)^{2/p}}{3c_b}}$, 
  and let $\tilde{X}: = X(\theta(t_0))$. We have, by the above,
  $$\Prob\big\{\|\tilde{X}\|_2-\Med\,\|\tilde{X}\|_2\geq t_0 \big\}\geq \frac{1}{C} \exp\bigg(
- C\cdot \frac{t_0^2}{K^2 \big(\log\big(2+\frac{ K^2n }{ t_0^2 }\big) \big)^{2/p} }\bigg).$$  
When $n$ is greater than a sufficiently large constant,  
$$\frac{t_0^2}{K^2 \big(\log\big(2+\frac{ K^2n }{ t_0^2 }\big) \big)^{2/p} }
  = \frac{(\log n)^{2/p}}{3c_b  \Big(\log \Big( 2 + \frac{3c_b n}{(\log n)^{2/p}}\Big)\Big)^{2/p}}
  \le \frac{(\log n)^{2/p}}{3c_b  \big(\log( \sqrt{n}) \big)^{2/p}}
  \le \frac{2}{3c_b},
$$
where we used that $p \ge 1$. 
We conclude that for $t \in [0, t_0]$,
$$\Prob\big\{\|\tilde{X}\|_2-\Med\,\|\tilde{X}\|_2\geq t \big\}\geq 
\Prob\big\{\|\tilde{X}\|_2-\Med\,\|\tilde{X}\|_2\geq t_0\big\}\geq 
\frac{1}{C} \exp\bigg(- \frac{2C}{3c_b} \bigg).$$
The lower tail is treated the same way. 
By adjusting the constant $C$, it implies the proposition for $t \in [0, t_0]$,
and completes the proof.
  \end{proof}

\section{Proof of Theorem~\ref{1-3974-1948798}}\label{1876491874968}

Our proof of Theorem~\ref{1-3974-1948798} is based on a modification of {\it the induction method} of Talagrand.
In fact, the first part of the proof which deals with setting up a recursive relation for a modified convex distance,
essentially repeats, up to minor changes, the standard account of the method (see, for example, \co{ \cite[p.~72-79]{Ledoux}}).

\co{We recall that {\it Talagrand's convex distance} between
a point $x\in\R^n$ and a set $A\subset\R^n$ is given by
$$
\max\limits_{a:\,\|a\|_2=1}\min\limits_{y\in A}\sum_{i=1}^n a_i\,{\bf 1}_{\{x_i\neq y_i\}}.
$$
Since we work with measures with (possibly) unbounded supports,
it is crucial for us to track the ``quantitative'' distance between $x_i$ and $y_i$, $i\leq n$, and to consider the differences $|x_i- y_i|$ instead of the indicators ${\bf 1}_{\{x_i\neq y_i\}}$.}

\begin{defi}
Given a point $x\in\R^n$ and a non-empty subset $A$
of $\R^n$, we define the {\bf modified convex distance} between $x$ and $A$ as
$$
\dist^c(x,A):=\max\limits_{a:\,\|a\|_2=1}\min\limits_{y\in A}\sum_{i=1}^n a_i\,|x_i- y_i|.
$$
\end{defi}

Given a non-empty $A\subset\R^n$ and $x\in \R^n$, we denote by $U(x,A)$ the set of all vectors in $\R^n_+$
of the form
$$
U(x,A):=\bigg\{\Big(|x_i- y_i|\Big)_{i=1}^n:\;y\in A\bigg\},
$$
and let $V(x,A)\subset\R^n$ be the convex \co{hull} of $U(x,A)$. 

\begin{lemma} \label{lem: distanceConvex}
We have
\begin{align} \label{eq: convexDistance}
\dist^c(x,A)=\dist(0,V(x,A)), 
\end{align}
where the distance on the right hand side is the usual Euclidean distance in $\R^n$. Furthermore, when $A$ is convex, 
\begin{align} \label{eq: convexDistanceConvexSet}
    \dist^c(x,A) = \dist(x,A). 
\end{align}
\end{lemma}
\begin{proof}
\co{The first assertion of the
  lemma can be derived following Talagrand's treatment for the original convex distance
  (see, in particular, \cite[p.~72-73]{Ledoux}).}

We will provide the proof for the second assertion of the lemma for \co{reader's} convenience.
Let $A$ be a non-empty convex set. \co{Without loss of generality, $A$ is closed, and $x\notin A$.
By a compactness argument, there is a vector $y \in x-A$ with $\|y\|_2=\dist(0,x-A)=\dist(x,A)$.
The extremal property of $y$ implies that for all $ z \in x-A$, we have $z \cdot y \ge y \cdot y$.}

Now, for any $z \in \mathbb{R}^n$, let $\tilde z$ be the vector obtained from $z$
by replacing each component of $z$ by its absolute value. For each point $z' \in U(x,A)$, there exists $z \in x-A$ such that $z' = \tilde z$. Since $ \tilde z \cdot \tilde y \ge z \cdot y \ge \|y\|_2^2$, the set $U(x,A)$ is contained in 
the half-space $\{ w\in\R^n\,:\, w \cdot \tilde y \ge \|y\|_2^2 \}$, and the same is true for its convex \co{hull} $V(x,A)$.
\co{Therefore, $\dist(0,V(x,A)) \geq \| \tilde y\|_2 = \|y \|_2$.
On the other hand, since $x-y\in A$, we have $\tilde y \in U(x,A)\subset V(x,A)$,
and therefore $\dist(0,V(x,A))\leq \| \tilde y\|_2 = \|y \|_2$.
We conclude that $\dist(0,V(x,A))=\|y \|_2$, and the result follows.}
\end{proof}

The main technical result in this section is the following proposition.
\begin{prop}\label{prop: choiceOfL}
Let $K>0$, and let $\mu_1,\mu_2,\dots,\mu_{n}$ be $K$-subgaussian probability measures in $\R$.
Let $X=(X_1,X_2,\dots,X_{n})$ be distributed in $\R^{n}$ according to $\mu_1\times\mu_2\times\dots\times
\mu_{n}$, and let $A\subset\R^{n}$
be a non-empty Borel subset. 
Then, for any $\delta \in (0, \frac{1}{2} ]$, 
$$
\Exp\exp\bigg(\frac{\tilde c\,(\dist^c(X,A))^2}{K^2\log\big(2+\frac{n}{\log(2+ 1/\delta )}\big)}\bigg)
\leq \frac{4}{\Prob\big\{X\in A\big\} \delta },
$$
where $\tilde c>0$ is a universal constant.  
\end{prop}

Before we consider the proof, let us show how to derive Theorem~\ref{1-3974-1948798} from the above proposition. 

\begin{proof}[Proof of Theorem \ref{1-3974-1948798}]
First, note that it is sufficient to prove the statement for $t \ge C'K\,\sqrt{\log n }$ for a large constant $C'>1$.
For the upper tail, we let $A:= \{ x \in \mathbb{R}^n\,:\, f(x)\le \Med\,f(X) \}$. 
By Proposition~\ref{prop: choiceOfL}, 
for any $\delta \in (0, \frac{1}{2}]$, 
$$
  \Exp\exp\bigg(\frac{\tilde c\,(\dist^c(X,A))^2}{ K^2\log\big(2+\frac{n}{\log(2+ 1/ \delta )}\big)}\bigg)
  \leq \frac{8}{ \delta }.
$$
Let $A_t:=\{ x \in \mathbb{R}^n:\; {\rm dist}^c(x,A) < t \}$.
Observe that, since $f$ is convex, so is the set $A$, and therefore $A_t=\{x \in \mathbb{R}^n:\; {\rm dist}(x,A) < t\}$,
in view of Lemma~\ref{lem: distanceConvex}.
Applying Markov's inequality, we get 
\begin{align*}
  \mathbb{P}\{ f(X) \geq \Med\,f(X) + t \} 
& \le \mathbb{P}\{ X \notin A_t \}  \\
& \le  \frac{8}{ \delta } \exp\bigg(
    - \frac{\tilde c\, t^2}{ K^2 \log\big(2+\frac{n}{\log(2+ 1/ \delta )}\big)}
  \bigg).
\end{align*}

We choose $ \delta := \exp\big( - \frac{ \tilde c \, t^2/4}{ K^2\log( 2 + \frac{ K^2 n}{\tilde c \, t^2/4})} \big)$
(we can assume that $\delta\leq 1/2$ if $C'$ is sufficiently large).
Observe that $ \log( 2 + \frac{1}{ \delta}) \ge \log(1/ \delta) = \frac{  \tilde c \, t^2/4}{ K^2 \log( 2 + \frac{ K^2 n}{ \tilde c \, t^2/4})}$, and hence
$$ 
  \log\Big(2 + \frac{n}{\log(2+1/ \delta)}\Big) 
\le  
  \log \bigg(2 + \frac{ K^2 n}{ \tilde c \, t^2/4} \log\Big( 2+ \frac{ K^2 n}{ \tilde c \, t^2/4}\Big)\bigg)
\le 
  2\log \Big(2 + \frac{ K^2 n}{\tilde c \, t^2/4}\Big). 
$$ 
Therefore,
$$
  \mathbb{P}\{ f(X) \ge \Med\,f(X) + t \} 
\le 
  8 \exp \bigg( 
    \frac{\tilde c \, t^2/4}{ K^2\log( 2 + \frac{ K^2 n}{\tilde c \, t^2/4})} 
    - \frac{\tilde c\, t^2}{2 K^2 \log\big(2+\frac{ K^2 n}{ \tilde c \, t^2/4 }\big)}
  \bigg)
=
  8 \exp \bigg( - \frac{ c t^2}{ K^2 \log(2+ \frac{ K^2 n}{ c t^2 })} \bigg),
$$
where $c:=\frac{1}{4} \tilde c$. By assuming $C'>1$ to be sufficiently large
and recalling that $t\geq C'\,K\sqrt{\log n}$, we get 
$$
8 \exp \bigg( - \frac{ c t^2}{ K^2 \log(2+ \frac{ K^2 n}{ c t^2 })} \bigg)
\le \exp \bigg( - \frac{  ct^2/2}{ K^2 \log(2+ \frac{ K^2 n}{ t^2 })} \bigg),
$$
which completes treatment of the upper tail.

\medskip

For the lower tail, we take $A:= \{ x \in \mathbb{R}^n\,:\, f(x) \le \Med\,f(X) - t\}$
and define $A_t:=\{ x \in \mathbb{R}^n:\; {\rm dist}^c(x,A) < t \}=
\{x \in \mathbb{R}^n:\; {\rm dist}(x,A) < t\}$ (with the last equality due to convexity of $A$). 
Then $ \{ x \in \mathbb{R}^n\,:\, f(x) \ge \Med\,f(X) \} \subset  A_t^c $ and therefore
$ \mathbb{P}\{X\in A_t^c \} \ge \frac{1}{2}$. For $ \delta \in (0, \frac{1}{2}]$,
we have, in view of Proposition~\ref{prop: choiceOfL} and Markov's inequality,
\begin{align*}
  \frac{1}{2} \le \mathbb{P}\{X\in A_t^c \} 
& \leq   \frac{4}{ \mathbb{P}\{ X \in A\} \delta }
  \exp\bigg( - \frac{\tilde c\, t^2}{\log\big(2+\frac{n}{\log(2+ 1/ \delta )}} \bigg),
\end{align*}
which implies 
\begin{align*}
  \mathbb{P}\{ f(X) \le \Med\,f(X) - t \} 
& = \mathbb{P}\{ X \in A \}  \\
& \le  \frac{8}{ \delta} \exp\bigg(
    - \frac{\tilde c\, t^2}{\log\big(2+\frac{n}{\log(2+ 1/ \delta )}\big)}
  \bigg).
\end{align*}
Now, the same choice of $\delta$ leads to the desired bound.  
\end{proof}

\bigskip

As we have mentioned above, the proof of Proposition~\ref{prop: choiceOfL} is based on the \co{induction on dimension.} 
The next proposition sets up the argument.
\begin{prop}\label{198471094871098}
Let $n\geq 1$, and let $\mu_1,\mu_2,\dots,\mu_{n+1}$ be probability measures in $\R$.
Let $A\subset\R^{n+1}$
be a non-empty subset,
and for each $\alpha\in\R$, denote
$$
A(\alpha):=\big\{v\in \R^n:\;(v,\alpha)\in A\big\}.
$$
Let $X=(X_1,X_2,\dots,X_{n+1})$ be distributed in $\R^{n+1}$ according to $\mu_1\times\mu_2\times\dots\times
\mu_{n+1}$, and $X'$ be the vector of first $n$ components of $X$.
Then for every $\kappa>0$,
\begin{align*}
\Exp&\exp\big(\kappa\cdot (\dist^c(X,A))^2\big)\\
&\co{\leq \Exp_{X_{n+1}}\inf\limits_{\nu}\Bigg[ \exp\Bigg(
\kappa\cdot\bigg(\int\limits_\R|X_{n+1}-\alpha|\,d\nu(\alpha)\bigg)^2+
\int_{\mathbb{R}}\log\Big(\Exp_{X'}\,\exp\big(\kappa\cdot(\dist^c(X',A(\alpha)))^2\big)\Big)d\nu(\alpha)\Bigg)\Bigg], 
}
\end{align*}
where the infimum is taken over all discrete probability measures $\nu$ in $\R$ with a finite support.
\end{prop}
\begin{proof}
Take arbitrary element $(x,s)\in\R^n\times \R=\R^{n+1}$.
Observe that
$$U\big((x,s),A\big)=\bigcup\limits_{\alpha\in\R:\,A(\alpha)\neq\emptyset}\Big(U\big(x,A(\alpha)\big)\oplus (|s-\alpha|)\Big),
$$
where the notation ``$\oplus$'' should be understood as vector-wise concatenation producing vectors in $\R^{n+1}$.
Therefore, every vector of the form
$$
\int\limits_\R \big(v(\alpha)\oplus (|s-\alpha|)\big)\,d\nu(\alpha)=\bigg(\int\limits_\R v(\alpha)\,d\nu(\alpha),
\int\limits_\R|s-\alpha|\,d\nu(\alpha)\bigg)\in\R^{n+1},
$$
where $v(\alpha)\in V(x,A(\alpha))$, $\alpha\in\R$, and $\nu$ is a discrete probability measure
on $\R$ with a finite support, belongs to the convex hull $V((x,s),A)$ of $U((x,s),A)$.

Further, we have for every Borel probability measure $\nu$ on $\R$ and every choice of $v(\alpha)\in V(x,A(\alpha))$:
$$
\bigg\|\int\limits_\R v(\alpha)\,d\nu(\alpha)\bigg\|_2^2  
\co{ = \sum_{i=1}^n \bigg( \int_{\mathbb{R}} (v(\alpha))_i {\rm d}\nu(\alpha)\bigg)^2 }
\leq \int\limits_\R  \|v(\alpha)\|_2^2\,d\nu(\alpha),
$$
by Jensen's inequality.
Hence,
\begin{align*}
\bigg\|\bigg(\int\limits_\R v(\alpha)\,d\nu(\alpha),
\int\limits_\R|s-\alpha|\,d\nu(\alpha)\bigg)\bigg\|_2^2
\leq \int\limits_\R  \|v(\alpha)\|_2^2\,d\nu(\alpha)+\bigg(\int\limits_\R|s-\alpha|\,d\nu(\alpha)\bigg)^2.
\end{align*}
Recall \co{ from \eqref{eq: convexDistance}} that $\dist^c((x,s),A)=\dist(0,V((x,s),A))$ and $\dist^c(x,A(\alpha))=\dist(0,V(x,A(\alpha)))$.
Thus, taking $v(\alpha)\in V(x,A(\alpha))$ so that $\|v(\alpha)\|_2= \dist^c(x,A(\alpha))$ for all $\alpha$,
we obtain that
\begin{align} \label{eq: distInduction}
 \big(\dist^c((x,s),A)\big)^2
\leq \inf\limits_{\nu}\bigg(
\int\limits_\R  \big(\dist^c(x,A(\alpha))\big)^2\,d\nu(\alpha)
+\bigg(\int\limits_\R|s-\alpha|\,d\nu(\alpha)\bigg)^2\bigg),   
\end{align}
where the infimum is taken over all discrete probability measures $\nu$ on $\R$ with a finite support.
Clearly,
$$
\Exp\exp\big(\kappa\cdot (\dist^c(X,A))^2\big)
=\Exp_{X_{n+1}}\Exp_{X'}\exp\big(\kappa \cdot(\dist^c((X',X_{n+1}),A))^2\big).
$$
Further, applying \co{\eqref{eq: distInduction}} we get
\begin{align} 
  \nonumber 
\Exp_{X'}&\exp\big(\kappa \cdot(\dist^c((X',X_{n+1}),A))^2\big)\\
\nonumber
&\leq
\Exp_{X'}\,\inf\limits_{\nu}\exp\bigg(\kappa\cdot
\int\limits_\R  \big(\dist^c(X',A(\alpha))\big)^2\,d\nu(\alpha)
+\kappa\cdot\bigg(\int\limits_\R|X_{n+1}-\alpha|\,d\nu(\alpha)\bigg)^2\bigg)\\
\nonumber 
&\leq \inf\limits_{\nu}\,\Exp_{X'}\,\exp\bigg(\kappa\cdot
\int\limits_\R  \big(\dist^c(X',A(\alpha))\big)^2\,d\nu(\alpha)
+\kappa\cdot\bigg(\int\limits_\R|X_{n+1}-\alpha|\,d\nu(\alpha)\bigg)^2\bigg)\\
&=\inf\limits_{\nu}\,\bigg[ \exp\bigg(
\kappa\cdot\bigg(\int\limits_\R|X_{n+1}-\alpha|\,d\nu(\alpha)\bigg)^2\bigg)
\;
\Exp_{X'}\,\exp\bigg(\kappa\cdot
\int\limits_\R  \big(\dist^c(X',A(\alpha))\big)^2\,d\nu(\alpha)\bigg)\bigg].
\label{eq: convexDistance00}
\end{align}
\co{ 
We write
$$
\Exp_{X'}\,\exp\bigg(\kappa\cdot
\int\limits_\R  \big(\dist^c(X',A(\alpha))\big)^2\,d\nu(\alpha)\bigg)
= 
\mathbb{E}_{X'}  \prod_{\alpha}
\exp\bigg(\kappa\cdot
\big(\dist^c(X',A(\alpha))\big)^2\,\bigg)  ^{\nu\{\alpha\}},
$$
where the product is taken over all $\alpha$ in the support
of $\nu$ (which is a finite set in $\mathbb{R}$),
and $\nu\{\alpha\}$ is the probability mass of $\alpha$.   
Since $\sum_{\alpha} \nu\{\alpha\}=1$,} in view of Holder's inequality, the quantity \co{in \eqref{eq: convexDistance00} }is majorized by
$$
\inf\limits_{\nu}\,\bigg[ \exp\bigg(
\kappa\cdot\bigg(\int\limits_\R|X_{n+1}-\alpha|\,d\nu(\alpha)\bigg)^2\bigg)
\;
\prod\limits_{\alpha\in\R}\Big(\Exp_{X'}\,\exp\big(\kappa\cdot\big(\dist^c(X',A(\alpha))\big)^2\big)\Big)^{\co{\nu\{\alpha\}}}\bigg],
$$
and the result follows.
\end{proof}

\begin{rem}
The class of measures $\nu$ in the above proposition is restricted to discrete measures
to avoid any discussion of measurability.
\end{rem}

By considering two-point probability measures $\nu$ of the form $\lambda\delta_{X_{n+1}}+(1-\lambda)\delta_y$,
from the last proposition we get the following corollary.
\begin{cor}\label{1-2981-9585-9u}
Let $A$, $X$, $X'$ and $\kappa$ be as in Proposition~\ref{198471094871098}. Then
\begin{align*}
&\Exp_{\co{X} }\exp\big(\kappa\cdot (\dist^c(X,A))^2\big)\\
&\leq \Exp_{\co{X_{n+1}}} \inf\limits_{\nu=\lambda\delta_{X_{n+1}}+(1-\lambda)\delta_y,\,\lambda\in[0,1],\,y\in\R}\bigg[ \exp\bigg(
\kappa\cdot\bigg(\int\limits_\R|X_{n+1}-\alpha|\,d\nu(\alpha)\bigg)^2\bigg)\\
&\hspace{3cm}\cdot
\co{ \exp\bigg( \int_{\alpha \in \mathbb{R}}
\log 
\big(\Exp_{ X'}\,\exp\big(\kappa\cdot(\dist^c(X',A(\alpha)))^2\big)\big)d\nu(\alpha)\bigg)}\bigg]\\
&=\Exp_{\co{X_{n+1}}}\inf\limits_{
\co{ \lambda } \in[0,1],\,y\in\R}\bigg[\exp\bigg(
-\lambda\,\log\frac{1}{\Exp_{\co{ X'}}\,\exp\big(\kappa\cdot(\dist^c(X',A(X_{n+1})))^2\big)}\\
&\hspace{3cm}-(1-\lambda)\log\frac{1}{\Exp_{\co{X'}}\,\exp\big(\kappa\cdot(\dist^c(X',A(y)))^2\big)}
+\kappa\cdot(X_{n+1}-y)^2\,(1-\lambda)^2\bigg) \bigg].
\end{align*}
\end{cor}

\bigskip

Next, we record the following elementary fact.
\begin{lemma}
\label{lemma: minBasic}
Let $-\infty\le b\le a<+\infty$, and let $c_{0}>0$, $R>0$. Then
\begin{align}
\min_{\lambda\in[0,1]}\Big(-\lambda b-(1-\lambda)a+c_{0}R^{2}(1-\lambda)^{2}\Big)=\begin{cases}
-a+c_{0}R^{2}, &\mbox{ if }(a-b)\ge2c_{0}R^{2}\\
-b-\frac{(a-b)^{2}}{4c_{0}R^{2}}, &\mbox{ if }(a-b)\le2c_{0}R^{2}.
\end{cases}\label{eq: minValue}
\end{align}
\end{lemma}
\begin{proof}
We have
\begin{align*}
\min_{\lambda\in[0,1]}\bigg(-\lambda b-(1-\lambda)a+c_{0}R^{2}(1-\lambda)^{2}\bigg)= & -a+\min_{\lambda\in[0,1]}\bigg(c_{0}R^{2}(1-\lambda)^{2}+\lambda(a-b)\bigg)\\
= & -a+c_{0}R^{2}\min_{\lambda\in[0,1]}\bigg(1+\Big(\frac{a-b}{c_{0}R^{2}}-2\Big)\lambda+\lambda^{2}\bigg).
\end{align*}
The expression $\big(1+\big(\frac{a-b}{c_{0}R^{2}}-2\big)\lambda+\lambda^{2}\big)$, $\lambda\in[0,1]$,
is minimized at $\lambda=\max(0,1-\frac{a-b}{2c_{0}R^{2}})$. And
\eqref{eq: minValue} follows since
\begin{align*}
1+\bigg(\frac{a-b}{c_{0}R^{2}}-2\bigg)
\bigg(1-\frac{a-b}{2c_{0}R^{2}}\bigg)+
\bigg(1-\frac{a-b}{2c_{0}R^{2}}\bigg)^{2}
=1-\bigg(1-\frac{a-b}{2c_{0}R^{2}}\bigg)^{2}=\frac{a-b}{c_{0}R^{2}}-\frac{(a-b)^{2}}{4c_{0}^{2}R^{4}}.
\end{align*}
\end{proof}

\bigskip

As an immediate consequence of Corollary~\ref{1-2981-9585-9u} and Lemma~\ref{lemma: minBasic},
by considering two-point measures \co{ we obtain the following proposition.  }
\begin{prop}\label{3-91-409184-018}
Let $n\geq 1$, and let $\mu_1,\mu_2,\dots,\mu_{n+1}$ be probability measures in $\R$.
Let $A\subset\R^{n+1}$
be a non-empty subset,
and for each $\alpha\in\R$, denote
$$
A(\alpha):=\big\{v\in \R^n:\;(v,\alpha)\in A\big\}.
$$
Let $X=(X_1,X_2,\dots,X_{n+1})$ be distributed in $\R^{n+1}$ according to $\mu_1\times\mu_2\times\dots\times
\mu_{n+1}$, and $X'$ be the vector of first $n$ components of $X$.
Then for every $\kappa>0$,
\begin{align*}
\Exp\exp\big(\kappa\cdot (\dist^c(X,A))^2\big)
\leq \Exp\inf\limits_{y\in\R}\;\exp\big(H(X_{n+1},y)\big),
\end{align*}
where
$$
H\left(t,y\right):=\min\limits_{\lambda\in[0,1]}\big(-\lambda h(t)-(1-\lambda)h(y)+\kappa(1-\lambda)^{2}\,\left(y-t\right)^{2}\big),
$$
and $h:\R\to\R$ is any function satisfying
$$
h\left(x\right)\leq \log\frac{1}{\Exp\,\exp\big(\kappa\cdot(\dist^c(X',A(x)))^2\big)},\quad x\in\R.
$$
Moreover, the function $H\left(t,y\right)$ can be represented as
$$
H\left(t,y\right):=\begin{cases}-h(y)+
\kappa\left(y-t\right)^{2}, &\mbox{if } h\left(y\right)-h\left(t\right)\ge2\kappa\left(y-t\right)^{2},\\
-h\left(t\right)-\frac{\left(h\left(y\right)-h\left(t\right)\right)^{2}}{4\kappa\left(y-t\right)^{2}}, &\mbox{if }
0 \co{<} h\left(y\right)-h\left(t\right)\le2\kappa\left(y-t\right)^{2},\\
-h\left(t\right), &\mbox{if }h\left(y\right)-h\left(t\right)\leq 0.
\end{cases}
$$
\end{prop}

\begin{rem}\label{3-981740160987}
\co{Repeating the optimization argument from \cite[p.~74]{Ledoux}, we get for every pair numbers $t,\co{y}$ with
$h(y)\geq h(t)$, and for every number $Q\geq 4\kappa\,\left(y-t\right)^{2}$:
\begin{align*}
H\left(t,y\right)&= -h(y)+\min\limits_{\lambda\in[0,1]}\bigg(
4\kappa\,\left(y-t\right)^{2}\bigg((1-\lambda)^{2}/4-\lambda\,\frac{h(t)-h(y)}{4\kappa\,\left(y-t\right)^{2}}\bigg)\bigg)\\
&=-h(y)+4\kappa\,\left(y-t\right)^{2}\bigg(\frac{1}{4}\,{\bf 1}_{\big\{
\frac{h(t)-h(y)}{4\kappa\,\left(y-t\right)^{2}}\leq -1/2\big\}}
+\bigg(-\frac{h(t)-h(y)}{4\kappa\,\left(y-t\right)^{2}}-
\bigg(\frac{h(t)-h(y)}{4\kappa\,\left(y-t\right)^{2}}\bigg)^2\bigg)\,
{\bf 1}_{\big\{\frac{h(t)-h(y)}{4\kappa\,\left(y-t\right)^{2}}> -1/2\big\}}\bigg)\\
&\leq -h(y)+4\kappa\,\left(y-t\right)^{2}
\log\bigg(2-\exp\bigg(\frac{h(t)-h(y)}{4\kappa\,\left(y-t\right)^{2}}\bigg)\bigg)\\
&\leq -h(y)+Q\log\bigg(2-\exp\bigg(\frac{h(t)-h(y)}{Q}\bigg)\bigg),
\end{align*}
where in the second line we applied Lemma~\ref{lemma: minBasic} with
$a:=0$, $b:=\frac{h(t)-h(y)}{4\kappa\,\left(y-t\right)^{2}}$, and $c_{0}R^{2}:=\frac{1}{4}$, and
where in the last line we used that the function
$s\to s\log\big(2-\exp\big(\frac{h(t)-h(y)}{s}\big)\big)$, $s>0$, is non-decreasing.}
\end{rem}

\bigskip

The next lemma encapsulates the initial step of the induction:
\begin{lemma}\label{10861-5-1098-98-0}
Let $\mu$ be a $K$-subgaussian probability measure on $\R$, and $X$ be distributed according to $\mu$.
Then for any choice of the parameter $L\geq \sqrt{2}\,K$ and any non-empty Borel subset $A\subset \R$,
$$
\Exp\exp\bigg(\frac{(\dist^c(X,A))^2}{L^2}\bigg)\leq \frac{4}{\mu(A)}.
$$
\end{lemma}
\begin{proof}
\co{
Since $A$ is a subset of $\mathbb{R}$, the convex distance ${\rm dist}^c(\cdot, A)$ coincides with ${\rm dist}(\cdot, A)$.} Without loss of generality, the set $A$ is closed.
Let $x\in A$ be a point with
$\dist^c(0,A)=\dist(0,A)= \co{|x|}$.

Then
$$
\Exp\exp\bigg(\frac{(\dist^c(X,A))^2}{L^2}\bigg)
\leq \Exp\exp\bigg(\frac{2 X^2+2x^2}{L^2}\bigg)
\leq 2\exp\bigg(\frac{2x^2}{L^2}\bigg).
$$
It remains to
note that
$$\mu(A)\leq \Prob\{|X|\geq |x|\}=\Prob\{\exp(2X^2/L^2)\geq \exp(2x^2/L^2)\}\leq \exp(-2x^2/L^2)\,\Exp\,\exp(2X^2/L^2).$$
The result follows.
\end{proof}

In the next lemma, we deal with ``the main part'' of the induction argument.
The basic idea is to split the argument into two cases, according to how much of the ``total mass''
of a set $A$ is located far from the origin.
\begin{lemma}\label{1-9847109870}
Let $m\geq 2$, and let $A$ be a non-empty Borel subset of $\R^m$.
For each $x\in\R$, let
$$
A(x):=\big\{y\in \R^{m-1}:\;(y,x)\in A\big\}.
$$
Further, let $\mu_1,\mu_2,\dots,\mu_m$ be $K$-subgaussian measures on $\R$, each supported on finitely many points,
let $X=(X_1,\dots,X_m)$ be distributed according to $\mu_1\times\mu_2\times\dots\times\mu_m$, and let $X'$ be the vector
of first $m-1$ components of $X$. Assume that for some $R\geq 1$, $L\geq 16K$ and every $x\in\R$,
$$
\Exp\,\exp\bigg(\frac{(\dist^c(X',A(x)))^2}{L^2}\bigg)\leq \frac{R}{\Prob\{X'\in A(x)\}}.
$$
Then
$$
\Exp\exp\bigg(\frac{(\dist^c(X,A))^2}{L^2}\bigg)\leq\frac{R(1-\exp(-L^2/(64K^2)))^{-2}}{\Prob\{X\in A\}}.
$$
\end{lemma}
\begin{proof}
Define parameters $\tilde L:=L/4$ and $M:=L/(8K)$. \co{ Our goal is to show that
$$
\Exp\exp\bigg(\frac{(\dist^c(X,A))^2}{L^2}\bigg)\leq\frac{R(1-\exp(-M^2))^{-2}}{\Prob\{X\in A\}}.
$$
}

We consider two cases. First, assume that 
\co{ 
\begin{equation}\label{ljnfiejfnqipfnwp}
\Prob\{X\in A\mbox{ and }X_m\in [-\tilde L,\tilde L]\}
\geq (1-\exp(-M^2))\Prob\{X\in A\}.
\end{equation}
}In this case, we essentially repeat the standard ``induction method'' argument employed in the proof of dimension-free
subgaussian concentration on the cube.
Let $x_{b}$ be a point in $[-\tilde L,\tilde L]$ such that $\Prob\{X'\in A(x_{b})\}\geq \Prob\{X'\in A(x)\}$ for all $x\in[-\tilde L,\tilde L]$ (such a point $x_b$ exists since, by our assumption, $X'$ can take
only finitely many values and hence $\{\Prob\{X'\in A(x)\},\,x\in\R\}$ is a finite set).
In view of Proposition~\ref{3-91-409184-018},
\begin{align*}
\Exp\exp\big((\dist^c(X,A))^2/L^2\big)
\leq \Exp\,\exp\big(H(X_{m},x_b)\big),
\end{align*}
where
\begin{align*}
H\left(t,x_b\right):&=\min\limits_{\lambda\in[0,1]}\big(-\lambda h(t)-(1-\lambda)h(x_b)+(1-\lambda)^{2}\,\left(x_b-t\right)^{2}/L^2\big)\\
&=\begin{cases}-h(x_b)+
\left(x_b-t\right)^{2}/L^2, &\mbox{if } h\left(x_b\right)-h\left(t\right)\ge2\left(x_b-t\right)^{2}/L^2,\\
-h\left(t\right)-\frac{L^2\left(h\left(x_b\right)-h\left(t\right)\right)^{2}}{4\left(x_b-t\right)^{2}}, &\mbox{if }
0\leq h\left(x_b\right)-h\left(t\right)\le2\left(x_b-t\right)^{2}/L^2,\\
-h\left(t\right), &\mbox{if }h\left(x_b\right)-h\left(t\right)\leq 0,
\end{cases}
\end{align*}
and
\begin{equation}\label{1-398471-9487}
h\left(u\right):=\log\bigg(\frac{\Prob\{X'\in A(u)\}}{R}\bigg)\leq
\log\frac{1}{\Exp\,\exp\big((\dist^c(X',A(u)))^2/L^2\big)},\quad u\in\R.
\end{equation}
Using the definition of $x_b$, the equation $\frac{16\tilde L^{2}}{L^2}=1$, and Remark~\ref{3-981740160987} with parameters $Q:=1$ and $\kappa:=1/L^2$, we get
$$
H(X_{m},x_b)\leq -h(x_b)+\log\big(2-
\exp\big(h(X_{m})-h(x_b)\big)\big),\;
\mbox{ whenever }X_m\in[-\tilde L,\tilde L].
$$
On the other hand, for all realizations of $X_m\notin[-\tilde L,\tilde L]$ we can crudely bound the function as
$$
H(X_{m},x_b)\leq -h(x_b)+
\left(x_b-X_m\right)^{2}/L^2.
$$
Combining the relations, we get
\begin{align}
  \nonumber 
&\Exp\,\exp\big((\dist^c(X,A))^2/L^2\big)\\
\nonumber 
&\leq \Exp\bigg[\exp(-h(x_b))
\big(2-
\exp\big(h(X_{m})-h(x_b)\big)\big)
\,{\bf 1}_{\{X_m\in[-\tilde L,\tilde L]\}}\\
\nonumber 
&\hspace{1cm}+\exp\big(-h(x_b)+
\left(x_b-X_m\right)^{2}/L^2\big)\,{\bf 1}_{\{X_m\notin[-\tilde L,\tilde L]\}}
\bigg]\\
\label{eq: indc00}
&\co{ \leq \frac{R}{\Prob\{X'\in A(x_b)\}}\Exp\,\bigg(
\bigg(2-\frac{\Prob_{X'}\{X'\in A(X_m)\}}{\Prob\{X'\in A(x_b)\}}\bigg)
\,{\bf 1}_{\{X_m\in[-\tilde L,\tilde L]\}}+\exp\big(4 X_m^{2}/L^2\big)\,{\bf 1}_{\{X_m\notin[-\tilde L,\tilde L]\}}
\bigg)}, 
\end{align}
where the inequality follows from the definition of $h$
and the bound
$$\exp\big( (x_b-X_m)^{2}/L^2\big)\,{\bf 1}_{\{X_m\notin[-\tilde L,\tilde L]\}}
\leq \exp\big(4X_m^{2}/L^2\big)\,{\bf 1}_{\{X_m\notin[-\tilde L,\tilde L]\}}.$$
Observe that 
\begin{align} \nonumber 
  \Exp_{X_m}\,\bigg(2-\frac{\Prob_{X'}\{X'\in A(X_m)\}}{\Prob\{X'\in A(x_b)\}}\bigg)
\,{\bf 1}_{\{X_m\in[-\tilde L,\tilde L]\}}
 &= 2\Prob\{ |X_m|\le \tilde L \} - \frac{ \Prob\{X\in A, \mbox{ and } X_m\in [-\tilde L,\, \tilde L]\} }{ \Prob\{X'\in A(x_b)\}}  \\
& \le 2 - \frac{ \Prob\{X\in A, \mbox{ and } X_m\in [-\tilde L,\, \tilde L]\} }{ \Prob\{X'\in A(x_b)\}}\nonumber\\
&\le  \frac{ \Prob\{X'\in A(x_b)\} }{ \Prob\{X\in A, \mbox{ and } X_m\in [-\tilde L,\, \tilde L]\}}, \label{eq: indc01}
\end{align}
where the last inequality follows as $2-x \le \frac{1}{x}$ for any $x \in [0,1]$, and since 
\begin{align}\label{ofvqiyfvjcaajwftcwf}
\Prob\{X\in A, \mbox{ and } X_m\in [-\tilde L,\, \tilde L]\} 
= \mathbb{E}_{X_m} \big[ {\bf 1}_{\{ |X_m|\le L\}} \Prob_{X'}\{ 
X' \in A(X_m)\} 
\big]
\le \Prob\{ X' \in A(x_b)\}. 
\end{align}
Next, we estimate the second summand inside the expectation in \eqref{eq: indc00}.
We have
\begin{align*}
\Exp\,\big[\exp\big(4X_m^{2}/L^2\big)\,{\bf 1}_{\{|X_m|>\tilde L\}}\big] 
:= \Exp\, Z^{ 4K^2/L^2} {\bf 1}_{ \{ 
Z>  \exp( (\tilde L/ K)^2 )\} }  
\end{align*}
where 
$$ Z:= \exp(X_m^2/K^2)$$
is a non-negative random variable satisfying $\mathbb{E}Z\le 2$ since $X_m$ is $K$-subgaussian. As $L\ge 16K$ and by applying Holder's and Markov's inequalities, we get 
\begin{align}
  \nonumber 
\Exp\, Z^{ 4K^2/L^2} {\bf 1}_{ \{ 
Z>  \exp( (\tilde L/ K)^2 )\} }  
 &\le  
\Exp\, Z^{ 1/2} {\bf 1}_{ \{ 
Z>  \exp( (\tilde L/ K)^2 )\} } 
\le (\Exp\, Z)^{1/2} \big( \Prob\{ Z > \exp( (\tilde L/ K)^2 )\})^{1/2}  \\
& \le  2^{1/2} \cdot \big( 2 \exp( - (\tilde L/K)^2 )\big)^{1/2}
= 2 \exp( -2M^2) \le \exp(-M^2) \label{eq: indc02},
\end{align}
where the last inequality holds since $M\ge 2$. 

\co{
Hence, combining \eqref{ljnfiejfnqipfnwp}, \eqref{eq: indc00}, \eqref{eq: indc01}, \eqref{ofvqiyfvjcaajwftcwf}, and \eqref{eq: indc02},
we obtain
\begin{align*}
\Exp\,&\exp\big((\dist^c(X,A))^2/L^2\big)\\
&\leq \frac{R}{\Prob\{X\in A\mbox{ and }X_m\in[-\tilde L,\tilde L]\}}
+\frac{R\exp(-M^2)}{\Prob\{X'\in A(x_b)\}}\\
&\leq \frac{R(1-\exp(-M^2))^{-1}}{\Prob\{X\in A\}}
+\frac{R(1-\exp(-M^2))^{-1} \exp(-M^2)}{\Prob\{X\in A\}},
\end{align*}
implying the result.
}
\medskip

Now, consider the second case: 
$\Prob\{X\in A\mbox{ and }X_m\notin[-\tilde L,\tilde L]\}> \exp(-M^2)\Prob\{X\in A\}$.
Observe that since
$$
\Prob\{X\in A\mbox{ and }X_m\notin[-\tilde L,\tilde L]\}=\int_{\R\setminus[-\tilde L,\tilde L]}\Prob_{X'}\{X'\in A(s)\}\,d\mu_m(s)
=\Exp_{X_m}\big(\Prob_{X'}\{X'\in A(X_m)\}{\bf 1}_{\{|X_m|>\tilde L\}}\big),
$$
there must exist a point $x_t\in\R\setminus[-\tilde L,\tilde L]$ with
$\Prob_{X'}\{X'\in A(x_t)\}\geq 2\Prob\{X\in A\}\exp(2x_t^2/L^2)$. Indeed, if we assume the opposite
then, by the above, 
\begin{align*}
\exp(-M^2)\,\Prob\{X\in A\}&<\Prob\{X\in A\mbox{ and }X_m\notin[-\tilde L,\tilde L]\}
\leq
2\Prob\{X\in A\}\,\Exp_{X_m}\big(\exp(2X_m^2/L^2){\bf 1}_{\{|X_m|>\tilde L\}}\big)\\ 
&= \co{ 2 \Prob\{X \in A\}\, \Exp\, Z^{ 2K^2/L^2} {\bf 1}_{ \{ 
Z>  \exp( (\tilde L/ K)^2 )\} }
\le 4 \Prob\{X\in A\} \exp(-2M^2)\,\,\,  (\mbox{by \eqref{eq: indc02}})}
\end{align*}
leading to contradiction. 

Applying again Proposition~\ref{3-91-409184-018},
we can write
$$
\Exp\exp\big((\dist^c(X,A))^2/L^2\big)
\leq \Exp\,\exp\big(-h(x_t)+
\left(x_t-X_m\right)^{2}/L^2\big),
$$
where $h$ is given by \eqref{1-398471-9487}.
Hence,
\begin{align*}
\Exp\exp\big((\dist^c(X,A))^2/L^2\big)
&\leq \frac{R}{\Prob\{X'\in A(x_t)\}}\,\exp(2x_t^2/L^2)\,\Exp\,\exp(2X_m^2/L^2)\\
&\leq \co{ \frac{R}{2\Prob\{X\in A\}\exp(2x_t^2/L^2)}\exp(2x_t^2/L^2)\cdot \mathbb{E}\,\exp(X_m^2/K^2),}
\end{align*}
and the result follows.
\end{proof}

\medskip

\begin{proof}[Proof of Proposition \ref{prop: choiceOfL}]
  Let $\delta \in (0,\frac{1}{2}]$ which could be an $n$-dependent parameter. 
  Let us first assume that the probability measures $\mu_1,\dots,\mu_n$ are supported on finitely many points.
  Define a positive parameter $L$ via the relation 
  $$L^2=512K^2\log\bigg(2+\frac{n}{\log(2+1/\delta )}\bigg).$$
  \co{Clearly, $L \ge  16K$ which satisfies the assumptions of} both Lemmas~\ref{10861-5-1098-98-0} and~\ref{1-9847109870}.
  Hence, applying Lemma~\ref{10861-5-1098-98-0} and then 
  Lemma~\ref{1-9847109870} inductively $n-1$ times,
  we get
  $$
  \Exp\exp\big((\dist^c(X,A))^2/L^2\big)\leq 
  \frac{4(1-\exp(-L^2/(64K^2)))^{-2(n-1)}}{\Prob\{X\in A\}}.
  $$
  Note that
  \begin{align*}
  (1-\exp(-L^2/(64K^2)))^{-2(n-1)}
  &=\bigg(1-\bigg(2+\frac{n}{\log(2+1/ \delta  )}\bigg)^{-8}\bigg)^{-2(n-1)}\\
  &\leq \bigg(1-\bigg(2+\frac{n}{\log(2+1/ \delta  )}\bigg)^{-8}\bigg)^{-2n}\\
  &\leq \exp\bigg(4n\,\bigg(2+\frac{n}{\log(2+1/ \delta )}\bigg)^{-8}\bigg)\bigg),
  \end{align*}
  where we used that for any number $0<\tau<1/2$, $(1-\tau)^{-1}\leq \exp(2\tau)$.
  Hence,  
  \begin{align*}
  (1-\exp(-L^2/(64K^2)))^{-2(n-1)}
  &\leq \big(2 + 1/ \delta \big)^{4\big(2+\frac{n}{\log(2 + 1/ \delta )}\big)^{-7}}
  < (2+1/ \delta )^{\frac{1}{2}} 
  \leq 1/ \delta,
  \end{align*}
  since
  $ 1/ \delta  \ge 2$, and the statement follows.

  Next, by an approximation argument we extend the proof to the setting when
  the supports of $\mu_1,\dots,\mu_n$ may be infinite.
  Assume that $A$ is open.
  For every $\varepsilon>0$ there exist finitely supported $K$--subgaussian
  measures $\mu_{\varepsilon,1},\dots,\mu_{\varepsilon,n}$ such that a vector $X_\varepsilon$
  distributed according to $\mu_{\varepsilon,1}\times\dots\times\mu_{\varepsilon,n}$,
  satisfies $\Prob\{X\in A\}\leq (1-\varepsilon)\Prob\{X_\varepsilon\in A\}$, and 
  $$
  \Exp\exp\bigg(\frac{\tilde c\,(\dist^c(X_\varepsilon,A))^2}{K^2\log\big(2+\frac{n}{\log(2+ 1/\delta )}\big)}\bigg)\geq (1-\varepsilon)\,\Exp\exp\bigg(\frac{\tilde c\,(\dist^c(X,A))^2}{K^2\log\big(2+\frac{n}{\log(2+ 1/\delta )}\big)}\bigg).
  $$
  Using the previously obtained result for finite measures and letting $\varepsilon\to0$,
  we derive the required statement for all open subsets of $\R^n$. Finally, approximating arbitrary non-empty
  $A$ with open sets $B\supset A$, we get the result.
\end{proof}

\begin{rem}
\co{As we already mentioned, our proof of the main result
uses a modified convex distance which is crucial in dealing with unbounded random variables.
The second main feature of our approach, compared to the original
argument of Talagrand, is that we estimate the product
$\Prob\{X\in A\}\,\Exp\exp\big((\dist^c(X,A))^2/L(\delta)^2\big)$ from above
by the quantity $1/\delta$ depending on $n$ and $t$ rather than by a universal constant.
The parameter $\delta$ introduces the necessary additional flexibility.}
\end{rem}

\end{document}